\documentclass{article}
\usepackage{amsfonts, amsmath,amscd, amssymb, latexsym, mathrsfs, verbatim, wasysym}
\usepackage{amssymb}
\usepackage{pb-diagram}

\usepackage{hyperref}
\usepackage[dvips]{graphicx}
\usepackage{epsfig}
\usepackage{psfrag}
\usepackage{captcont}
\usepackage[colorinlistoftodos]{todonotes}
\usepackage[english]{babel}
\usepackage[utf8x]{inputenc}
\usepackage[all]{xy}
\usepackage{amscd}
\usepackage{tikz-cd}
\usepackage{appendix}
\usepackage{amsthm}
\usepackage{appendix}
\usepackage{mathtools}
\usepackage[dvips]{graphicx}
\usepackage{amsfonts}
\usepackage{amssymb}

\usepackage{epsfig}
\usepackage{psfrag}
\usepackage{subfigure}
\usepackage{captcont}

\usepackage{color}
\definecolor{darkgreen}{cmyk}{1,0,1,.2}
\definecolor{m}{rgb}{1,0.1,1}
\definecolor{green}{cmyk}{1,0,1,0}
\definecolor{darkred}{rgb}{0.55, 0.0, 0.0}
\definecolor{test}{rgb}{1,0,0}
\definecolor{cmyk}{cmyk}{0,1,1,0}

\numberwithin{diagram}{section}
\numberwithin{equation}{section}

\newtheorem{Equation}{}[section]

\newtheorem{theorem}[Equation]{Theorem}
\newtheorem{proposition}[Equation]{Proposition}
\newtheorem{lemma}[Equation]{Lemma}
\newtheorem{corollary}[Equation]{Corollary}

\newtheorem{definition}[Equation]{Definition}

\newtheorem{remark}[Equation]{Remark}

\def\Av{\operatorname{Av}}

\def\diag{\operatorname{diag}}

\def\Prop{\operatorname{Prop}}

\def\Im{\operatorname{Im}}

\def\Ker{\operatorname{Ker}}

\def\Proj{\operatorname{Proj}}

\def\Supp{\operatorname{Supp}}
\def\Ad{\operatorname{Ad}}

\def\id{\operatorname{id}}

\def\supp{\operatorname{supp}}

\def\maE{\mathcal{E}}
\def\maH{\mathcal{H}}
\def\maK{\mathcal{K}}

\def\maL{\mathcal{L}}
\def\maM{\mathcal{M}}

\def\maP{\mathcal{P}}

\def\del{\partial}

\def\C{\mathbb C}

\def\Z{\mathbb Z}
\def\N{\mathbb N}

\def\maE{{\mathcal E}}

\def\maM{{\mathcal M}}

\def\maH{{\mathcal H}}

\def\what{\widehat}

\usepackage{color}
\definecolor{darkgreen}{cmyk}{1,0,1,.2}
\definecolor{m}{rgb}{1,0.1,1}
\definecolor{green}{cmyk}{1,0,1,0}

\definecolor{test}{rgb}{1,0,0}
\definecolor{cmyk}{cmyk}{0,1,1,0}

\def\diam{\operatorname{diam}}

\def\redg{\operatorname{red}}

\begin{document}

\title{The Higson-Roe sequence for \'etale groupoids.\\ II. The universal sequence for Equivariant families}

\author{Moulay-Tahar Benameur\\{IMAG, Univ Montpellier, CNRS, Montpellier, France}\\\textit{Email address}: \texttt{moulay.benameur@umontpellier.fr} \and
Indrava Roy\\Institute of Mathematical Sciences, HBNI, Chennai, India\\\textit{Email address}:\texttt{indrava@imsc.res.in}}

\maketitle

\begin{abstract}
This is the second part of our series about the Higson-Roe sequence for \'etale groupoids. We devote this part to the proof of the universal $K$-theory surgery exact sequence which extends the seminal results of N. Higson and J. Roe to the case of transformation groupoids. In the process, we prove the expected functoriality properties as well as the Paschke-Higson duality theorem.
\end{abstract}



\section{Introduction}

We pursue in this paper our systematic investigation of the secondary invariants associated with groupoids. Our approach follows the deep program initiated by N. Higson and J. Roe in their seminal papers \cite{HR-I, HR-II, HR-III}. The present paper is the second of our series and is devoted to the statement of the functoriality properties for our dual Roe algebras as well as to the proof of the Paschke-Higson duality isomorphism. As a corollary of these constructions, we could obtain  the proof of the existence of the universal Higson-Roe sequence for our \'etale groupoids.  We have assumed   that our groupoid $G$ is the transformation groupoid $X\rtimes \Gamma$ associated with the action of the discrete countable (infinite) group $\Gamma$ on the metrizable space $X$. These transformation groupoids are known to be  generic  in the study of the secondary invariants of foliations and laminations, see for instance \cite{BN, BenameurPiazza, ConnesBook}. Many constructions are though ready to be generalized to any \'etale Hausdorff groupoid and the details of this extension will appear in a forthcoming paper. \\

In the first paper of this series  \cite{BenameurRoyI}, we have introduced the dual Roe algebras for \'etale groupoids and we have deduced the Higson-Roe exact sequence as well as its compatibility with the Baum-Connes and  the Paschke morphisms. With the proof of the Paschke isomorphism and of the functoriality of our algebras carried out in the present paper for transformation groupoids, we complete the picture and obtain  the $K$-theoretic surgery exact sequence for these groupoids.  Let us now explain more precisely our results.
For any proper  $\Gamma$-space $Z$ together with some usual data, we denote by $D^*_\Gamma (X; (Z, L^2Z\otimes \ell^2\Gamma^\infty))$ the dual Roe algebra obtained in this case following \cite{BenameurRoyI}  (see Section  \ref{GammaFamilies}).  We first prove the  functoriality  of these $D^*_\Gamma$ algebras  for continuous {\em{$\Gamma$-coarse}} maps as well as of their $C^*_\Gamma$ ideals for Borel {\em{$\Gamma$-coarse}} maps.  To this end, we have  used the notion of Roe-Voiculescu covering isometries described in an independent appendix. As a corollary, we can  introduce the allowed structure groups for the groupoid $X\rtimes \Gamma$ as follows:
$$
S_{*+1} (X\rtimes \Gamma) \; := \; \lim_{\stackrel{\longrightarrow}{Z\subset \underline{E}\Gamma}} K_{*}\left(D^*_\Gamma (X ; (Z, L^2Z\otimes \ell^2\Gamma^\infty))\right), \text{ for }*=0, 1\in \Z_2,
$$
where $ \underline{E}\Gamma$ is any locally compact universal space for proper $\Gamma$-actions. Along the process, our techniques also allow to define this $K$-theory group by using the $D^*_\Gamma$ algebra defined with respect to any fiberwise ample Hilbert module, although the isomorphism hence obtained is not natural. The above direct limit is taken as usual with respect to inclusion of $\Gamma$-invariant cocompact closed subspaces $Z\subseteq \underline{E}\Gamma$, and using the system of group morphisms
$$
i_{Z'\subset Z}^D : K_*(D^*_\Gamma (X; (Z', L^2Z'\otimes \ell^2\Gamma^\infty)))\stackrel{}{\longrightarrow} K_*(D^*_\Gamma (X; (Z, L^2Z\otimes \ell^2\Gamma^\infty))).
$$

This definition is in the spirit of the Baum-Connes assembly map
$$
\mu_*^\Gamma\; :\;  RK_* (X\rtimes \Gamma) \longrightarrow K_*( C(X)\rtimes \Gamma),
$$
which is conjectured to be an isomorphism. Recall that the left hand side  is the limit
$$
RK_* (X\rtimes \Gamma) \; := \;   \lim_{\stackrel{\longrightarrow}{Z\subset \underline{E}\Gamma}} KK_\Gamma^* (Z, X).
$$
The main result of  the present paper is thus the following generalization of the Higson-Roe universal exact sequence\\

{\bf{Theorem.}} \label{Surgeryseq} {\em{There exists a periodic six-term exact sequence:

$$
\begin{diagram}
       \node{K_0(C(X)\rtimes_{\redg}\Gamma)}\arrow{e}\node{S_1(X\rtimes\Gamma)}\arrow{e}\node{RK_1(X\rtimes\Gamma)}\arrow{s}\\
\node{RK_0(X\rtimes\Gamma)}\arrow{n}\node{S_0(X\rtimes\Gamma)}\arrow{w} \node{K_1(C(X)\rtimes_{\redg}\Gamma)}\arrow{w}
\end{diagram}
$$
where the vertical boundary maps are  the Baum-Connes assembly maps $(\mu^{BC}_i)_{i=0,1}$  for the groupoid $X\rtimes \Gamma$.}}\\

\medskip

{{In the first three sections of this paper, we  do not assume the (proper) action to be cocompact, nor do we assume any Lie structure on our groupoids, so our results are valid in a wide enough topological category in the spirit of the coarse approach to primary and secondary index invariants \cite{BGL, HRbook, HigsonPedersenRoe}. As explained above, using cocompact proper actions, we proved that the Higson-Roe exact sequence includes the classical Baum-Connes map for our groupoid. Under the extra assumption of a Lie structure on the groupoid, another exact sequence was obtained by V. Zenobi using pseudodifferential calculus on adiabatic deformations, in the spirit of the Connes' tangent groupoid approach, see \cite{Zenobi19}. When $X$ is reduced to the point for instance, the Zenobi exact sequence turns out to be isomorphic to the classical Higson-Roe exact sequence, see \cite{ZenobiPreprint}. For transformation Lie groupoids, it is easy to see for instance that when the group is torsion free, the  Zenobi sequence is again isomorphic to the one obtained in the present paper. It is thus an interesting task to compare, in this smooth case, our Higson-Roe exact sequence with Zenobi's sequence when modified to take torsion into account. }} \\

As a corollary of the above Theorem  \ref{Surgeryseq}, our structure groups $S_*(X\rtimes\Gamma)$ appear as the precise obstructions for the Baum-Connes conjecture to hold. In view of invariants of eta type, see again \cite{BenameurPiazza} or \cite{BGL}, the following  exact sequence will thus play an important part \cite{BenameurRoyJFA}
$$
K_0(C(X)\rtimes_{\redg}\Gamma)\rightarrow S_1(X\rtimes\Gamma)\rightarrow RK_1(X\rtimes\Gamma)\rightarrow K_1(C(X)\rtimes_{\redg}\Gamma)
$$
In order to achieve the proof of Theorem \ref{Surgeryseq}, an important step is the  Paschke-Higson duality theorem, which can be stated as follows:\\

{\bf{Theorem.}} \label{Paschkev2}
{\em{With the above notations and assuming the the proper $\Gamma$-space $Z$ is cocompact, the Paschke-Higson map defined in \cite{BenameurRoyI} gives group isomorphisms
$$
K_i(Q^*_\Gamma(X; (Z, \ell^2\Gamma^\infty\otimes L^2Z))) \xrightarrow{\maP} KK_{\Gamma}^{i+1}(Z, X), \quad i\in \Z_2.
$$}}

The groups $KK_{\Gamma}^{i+1}(Z, X)$ are the $KK$-groups and can be described in purely topological terms, see for instance \cite{BaumConnes, BaumConnesHigson, ConnesBook}.  Passing to the inductive limit, we obtain the allowed universal Paschke-Higson isomorphism which identifies the LHS group $RK_* (X\rtimes \Gamma)$ in the Baum-Connes map with an inductive limit of dual $C^*$-algebra $K$-theory groups. \\

The proof of  Theorem \ref{Surgeryseq} then relies on an inspection of the functoriality properties of all the involved morphisms. More precisely, the following cube is shown to be commutative for inclusions $Z'\hookrightarrow Z$ of cocompact closed $\Gamma$-subspaces, and for $i\in \Z_2$:\\

\begin{tiny}
\begin{tikzcd}[row sep=tiny, column sep=tiny]\hspace{-3cm}
& K_i(Q^*_\Gamma (X; (Z', L^2Z'\otimes \ell^2\Gamma^\infty))) \arrow[dl] \arrow[rr] \arrow[dd] & & K_i(Q^*_\Gamma (X; (Z, L^2Z\otimes \ell^2\Gamma^\infty))) \arrow[dl] \arrow[dd] \\
K_{i+1}(C^*_\Gamma (X; (Z', L^2Z'\otimes \ell^2\Gamma^\infty))) \arrow[rr, crossing over] \arrow[dd] & &  K_{i+1}(C^*_\Gamma (X; (Z, L^2Z\otimes \ell^2\Gamma^\infty)))\\
& KK^{i+1}_\Gamma (Z',X) \arrow[dl] \arrow[rr] & & KK^{i+1}_\Gamma (Z, X) \arrow[dl] \\
K_{i+1}(C(X) \rtimes_{\redg}\Gamma) \arrow[rr] & & K_{i+1}(C(X) \rtimes_{\redg}\Gamma) \arrow[from=uu, crossing over]\\
\end{tikzcd}
\end{tiny}

\medskip

\tableofcontents

\bigskip

{\bf{Notations.}}
For simplicity, all topological spaces used in this paper will be  locally compact Hausdorff and second countable, hence our spaces will always be paracompact. Given such space $Z$, we shall denote as usual by $C_0(Z)$ the $C^*$-algebra of complex valued continuous functions on $Z$ which vanish at infinity, while $C_c(Z)$ will be the subalgebra composed of the compactly supported functions. We shall also use the bigger multiplier algebra $C_b(Z)$ composed of the bounded continuous complex valued functions on $Z$.
For a given $C^*$-algebra $B$ and unless otherwise specified, all Hilbert $B$-modules  will be countably generated as $B$-modules. In particular, all Hilbert spaces will be countably generated.  Given  Hilbert $B$-modules $E$ and $E'$, we shall abusively denote by $\maL_B(E, E')$ the space of adjointable operators from $E$ to $E'$, so in particular such operators are $B$-linear and bounded. The subspace of $B$-compact operators will be denoted by $\maK_B(E, E')$, and when $E'=E$, we obtain the $C^*$-algebras which are rather denoted $\maL_B(E)$ and $\maK_B(E)$. Recall that  $\maK_B(E)$ is a closed two-sided involutive ideal in $\maL_B(E)$. When $B=\C$ is the $C^*$-algebra of the complex numbers, then it is simply dropped from the notation.  The notation $\maL (H)_{*-str}$ will  be used to emphasize that the space $\maL (H)$ is rather endowed with the $*$-strong topology. So, for instance and given a topological space $X$,  $C(X, \maL (H)_{*-str})$ is the space of continuous functions from $X$ to $\maL(H)$  endowed with the $*$-strong topology.
Given a group $\Gamma$ which acts on a set $\mathfrak{A}$, we shall denote as it is customary by $\mathfrak{A}^\Gamma$ the subset of $\mathfrak{A}$ composed of the $\Gamma$-invariant elements.\\

\medskip

{\em{Acknowledgements.}}
The authors  wish to thank P.S. Chakraborty, T. Fack, N. Higson, V. Mathai,  P. Piazza, B. Saurabh, G. Skandalis, R. Willett and V. Zenobi
for many helpful discussions. We are especially indebted to the referee for having read carefully this manuscript, and for her/his comments to improve it, especially the important suggestion of removing the cocompactness condition in all the statements of the first three sections.
MB thanks the French National Research Agency for support via the ANR-14-CE25-0012-01 (SINGSTAR).
IR thanks the Homi Bhabha National Institute and the Indian Science and Engineering Research Board via MATRICS project MTR/2017/000835 for support.

%


\section{Review of the Higson-Roe sequence}\label{GammaFamilies}

We review in this first section our results about dual algebras and the Higson-Roe sequence for {\'{e}}tale groupoids, see  \cite{BenameurRoyI}. The present paper concentrates on the case of \'etale groupoids which are associated with countable discrete group actions, so we briefly recall these results for such groupoids.

We consider a countable discrete group $\Gamma$ which acts on the right by homeomorphisms of the compact metrizable finite dimensional  space $X$. The groupoid $G$ is the action groupoid $X\rtimes \Gamma$ whose space of units is  $X$ and whose space of arrows is $X\times \Gamma$ with the following rules:
$$
s(x, \gamma) = x\gamma, \; r(x, \gamma) = x, \; \text{ and }\; (x, \gamma) (x', \gamma') = (x, \gamma\gamma') \text{ if } x'=x\gamma.
$$
The groupoid $G$   desingularizes   the space of leaves of a lamination which is constructed by suspending the action, through the so-called foliation monodromy groupoid, see for instance \cite{BenameurPiazza}.
Let $(Z, d)$ be a given locally compact proper-metric space which is endowed with an action of our group $\Gamma$. We assume furthermore that $\Gamma$ acts properly  on $Z$ and consider the $\Gamma$-space $Y= X\times Z$ which is then a proper $\Gamma$-space.

Let $(H, U)$ be a unitary Hilbert space representation of $\Gamma$ together with an ample $\Gamma$-equivariant {$C(X)$}-representation $\pi$ of {$C_0(Y)$}.
Recall that any adjointable operator $T$ of $ \maL_{C(X)} (C(X)\otimes H)$ is given by a field $(T_x)_{x\in X}$ of bounded operators on $H$ which is $*$-strongly continuous.
For instance, for {a general $C(X)$-representation $\pi$ of $C_0(Y)$ and} any $f\in C_0(Y)$, the operator $\pi (f)$ can be written as the $*$-strongly continuous field {$(\pi_x(f_x))_{x\in X}$} where each $\pi_x$ is a representation of {$C_0(Z)$} in the Hilbert space $H$. Moreover, it is easy to see that {$\pi_x(f_x)$} only depends on the restriction {$f_x$} of $f$ to $\{x\}\times Z$. An adjointable operator is  $\Gamma$-equivariant, if the field $(T_x)_{x\in X}$ satisfies the relations
$$
T_{x g} = U_g^{-1} T_x U_g, \quad  (x, g)\in X\times \Gamma.
$$
The space of $\Gamma$-equivariant adjointable operators is denoted as usual $\maL_{C(X)} (C(X)\otimes H)^\Gamma$. We denote by $D^*_\Gamma (X; (Z,H))$ and $C^*_\Gamma (X; (Z, H))$ the corresponding Roe algebras as defined in \cite{BenameurRoyI}, but for our groupoid $X\rtimes \Gamma$ and our specific Hilbert $G$-module $C(X)\otimes H$. More precisely,
$D^*_\Gamma (X; (Z,H))$ is defined as the norm closure in $\maL_{C(X)} (C(X)\otimes H)$ of the following space
$$
\{T\in \maL_{C(X)} (C(X)\otimes H)^\Gamma, T \text{ has finite propagation and }[T, \pi (f)]\in C(X, \maK(H)) \text{ for any }f\in C_0(Y)\}.
$$
The ideal $C^*_\Gamma (X; (Z, H))$ is composed of all the elements $T$ of $D^*_\Gamma (X; (Z,H))$ which satisfy in addition that
$$
T \pi (f) \in C(X, \maK(H)) \text{ for any }f\in C_0(Y).
$$
The finite propagation property here is supposed to hold uniformly on $X$, so $(T_x)_{x\in X}$ has finite propagation if there exists a constant $M\geq 0$ such that
for any $\varphi, \psi\in C_0(Z)$ with $d (\Supp (\varphi), \Supp (\psi)) > M$, we have
$$
\pi_x (\varphi) T_x \pi_x (\psi) = 0, \; \; \forall x\in X.
$$
We thus have the short exact sequence of $C^*$-algebras
$$
0 \to C^*_\Gamma (X; (Z, H)) \hookrightarrow D^*_\Gamma (X; (Z,H)) \longrightarrow Q^*_\Gamma (X; (Z,H)) \to 0,
$$
where we have denoted by $Q^*_\Gamma (X; (Z,H))$  the quotient $C^*$-algebra of $D^*_\Gamma (X; (Z,H))$ by its two-sided closed involutive  ideal $C^*_\Gamma (X; (Z, H))$.
Applying the topological $K$-functor, we end up with well defined boundary maps
$$
\partial_i \;:\; K_i\left(Q^*_\Gamma (X; (Z,H))\right) \longrightarrow K_{i+1}\left( C^*_\Gamma (X; (Z, H))\right), \quad i\in \Z_2.
$$
which fit in  the following  periodic six-term exact sequence of topological $K$-theory groups
\begin{displaymath}\label{Figure1}
\xymatrixcolsep{1pc}\xymatrix{
K_{*} (C^*_\Gamma (X; (Z, H))) \ar[rr]   &  & K_{*} (D^*_\Gamma (X; (Z,H)))   \ar[dl]^{}\\ & K_*(Q^*_\Gamma (X; (Z,H)))  \ar[ul]_{\partial_*}  }
\end{displaymath}
The Paschke-Higson map  can be described for our specific groupoid $G=X\rtimes \Gamma$ as the two group morphisms
$$
\maP_i^Z \;:\;  K_i\left( Q^*_\Gamma (X; (Z,H))\right) \longrightarrow KK^{i+1}_\Gamma ( Z, X), \quad i\in \Z_2,
$$
where $K_i( Q^*_\Gamma (X; (Z,H)))$ is $K$-theory of the $C^*$-algebra $Q^*_\Gamma (X; (Z,H))$ while $KK^{i+1}_\Gamma ( Z, X)$ is the $\Gamma$-equivariant $KK$-theory of the pair of $\Gamma$-algebras $C_0(Z), C(X)$ \cite{Kasparov2}.  It is defined as follows (see again   \cite{BenameurRoyI}).
For $i=0$ for instance and starting with a projection $E$ in $Q^*_\Gamma (X; (Z,H))$, the image of the class of $E$ under $\maP_0^Z$ is the class of the $\Gamma$-equivariant Kasparov odd cycle
$$
\left(C(X)\otimes H, \pi_Z, 2E-I\right),
$$
where the representation $\pi_Z$ is $\pi\circ p_2^*$ with $p_2^*:C_0(Z) \to C_0(Y)$ the morphism induced by the (proper) second projection $p_2:Y\to Z$. A similar definition yields for $i=1$ to the Paschke-Higson map $\maP_1$.
\begin{remark}\
According to \cite{BenameurRoyI}, the range of the Paschke map is expected to be the $G$-equivariant $KK$-theory of the pair of $G$-algebras $C_0(X\times Z), C(X)$.
However, it is easy to see that in our case, this latter group is isomorphic to the $\Gamma$-equivariant $KK$-theory of the pair of $\Gamma$-algebras $C_0(Z), C(X)$, i.e.
$$
KK^{*}_\Gamma ( Z, X)\simeq KK^{*}_{X\rtimes\Gamma} ( X\times Z, X).
$$
{{Notice that any $\Gamma$-equivariant Hilbert $C(X)$-module can be seen as a $X\rtimes\Gamma$-module. The previous isomorphism is obtained by using the obvious left action of $C(X)$ on any Hilbert $C(X)$-module, and reciprocally by tensoring with the unit of $C(X)$.}}
\end{remark}

An interesting representation $\pi$ is given by $L^2 (Z)=L^2 (Z, \mu_Z)$ for a choice of a Borel $\Gamma$-invariant measure $\mu_Z$ on $Z$, which we shall always assume to be fully supported.
Finally, recall also {{that in the case where the proper  $\Gamma$-space  $Z$ is cocompact}}, we have the classical Baum-Connes map   (see \cite{BaumConnes}):
$$
\mu_{\Gamma}^Z: KK^{*}_\Gamma ( Z, X)\longrightarrow  K_* (C(X)\rtimes_r \Gamma),
$$
where $C(X)\rtimes_r \Gamma$ is the reduced crossed product $C^*$-algebra. {{When $Z$ is not cocompact, one needs to replace the LHS by an inductive limit as we shall explain later on. The main  results that will be needed from the companion paper \cite{BenameurRoyI} concern the case of cocompact actions, and can be gathered as follows}}:

\begin{theorem}\cite{BenameurRoyI}
{If the $\Gamma$-proper space $Z$ is cocompact, then}
\begin{enumerate}
\item  The $C^*$-algebra $C^*_\Gamma (X; (Z, L^2Z\otimes \ell^2\Gamma))$ is Morita equivalent to the reduced crossed product $C^*$-algebra $C(X)\rtimes_r \Gamma$. In particular, we have  isomorphisms
$$
\maM_i: K_i (C^*_\Gamma (X; (Z, L^2Z\otimes \ell^2\Gamma)))\longrightarrow K_i(C(X)\rtimes_r \Gamma), \quad i=0, 1.
$$
\item For $i=0, 1$, the following diagram commutes
\[
\begin{CD}\label{BaumConnesassembly}
K_i(Q^*_\Gamma (X; (Z,L^2Z\otimes \ell^2\Gamma))) @> \del_i   >> K_{i+1}(C^*_\Gamma (X; (Z, L^2Z\otimes \ell^2\Gamma))) \\
@V\maP_i VV    @V\maM_{i+1}VV \\
KK^{i+1}_\Gamma ( Z, X) @> \mu_{\Gamma}^Z >> K_{i+1}(C(X)\rtimes \Gamma)\\
\end{CD}
\]

\end{enumerate}
\end{theorem}
\bigskip


\medskip

\section{Functoriality of dual algebras}

We proceed now to establish the functoriality of the $K$-theory groups of the Roe $C^*$-algebras $C^*_\Gamma (X; (Z, H))$, $D^*_\Gamma (X; (Z,H))$ and $Q^*_\Gamma (X; (Z,H))$,
corresponding to appropriate  classes of maps $f: Z'\rightarrow Z$. We first prove these {functoriality} results for  the $C^*$-algebras $C^*_\Gamma (X; (Z, H))$. Later on we shall show that under the extra continuity assumption  of the maps $f$, the functoriality properties hold as well for the $C^*$-algebras $D^*_\Gamma (X; (Z, H))$ and the quotient $C^*$-algebras $Q^*_\Gamma (X; (Z, H))$.

\subsection{Functoriality properties of the Roe ideal}

Recall that $X$ is a compact metrizable space (of finite dimension). We consider spaces $Y$ which are given as $Y=X\times Z$, for spaces $Z$ which are  proper-metric spaces on which
$\Gamma$ {{acts properly}}, as isometric homeomorphisms.  We always assume that our metric spaces are proper-metric spaces.
Recall that if $(Z, d)$ and $(Z', d')$ are metric spaces, then a map $f:Z'\rightarrow Z$ is metrically proper when the inverse images of bounded sets in {$Z$}  are bounded in {$Z'$}.
A metrically proper {Borel} map  $f:Z'\to Z$ is called a coarse map if given any $R>0$, there exists $S>0$ such that
 $$
 d' (z'_1,z'_2)\leq R \Longrightarrow d (f(z'_1),f(z'_2)) \leq S \quad \text{ for all } z_1',z_2' \in Z'.
 $$
Two coarse maps $f_1, f_2: Z'\rightarrow Z$ are called coarsely equivalent if there exists a constant $M>0$ such that
 $$
 d (f_1(z'), f_2(z')) \leq M \quad \text{ for all } z'\in Z'
 $$
 Suppose that the proper-metric spaces $(Z',d')$ and $(Z,d)$ are  endowed with the action of our countable discrete group $\Gamma$, again by isometries. Then a given map $f:Z'\to Z$ is coarsely equivariant if there exists a constant $M\geq 0$ such that
 $$
 d_Z (f(gz'), gf(z')) \leq M,\quad \forall g\in \Gamma\text{ and } \forall z'\in Z'.
 $$


\begin{definition}\label{Gcoarse}\label{coarse-equiv1}\ Suppose that the proper-metric spaces $(Z',d')$ and $(Z,d)$ are  endowed with the action of $\Gamma$ by isometries.
\begin{enumerate}
\item A metrically proper Borel map
$f:Z'\rightarrow Z$ is called a coarse $\Gamma$-map, if it is  coarse and coarsely equivariant.
 \item  The proper-metric $\Gamma$-spaces  $(Z',d')$ and $(Z,d)$  are $\Gamma$-coarsely equivalent if there exist coarse $\Gamma$-maps $f: Z'\rightarrow Z$ and
$g:Z\rightarrow Z'$ such that $f\circ g$ is coarsely equivalent to the identity on $Z$ and $g\circ f$ is coarsely equivalent to the identity on $Z'$.
 \end{enumerate}
\end{definition}

If $H$ is a given Hilbert space, we shall denote by $H^\infty$ the Hilbert space $\ell^2 (\N, H)\simeq H\otimes \ell^2\N$.
Given a $\Gamma$-equivariant faithful Hilbert space representation $(H, \pi)$ of $C_0(Z)$, we get an
ample $\Gamma$-equivariant representation $(\ell^2\Gamma\otimes H^\infty\simeq H\otimes \ell^2\Gamma^\infty, \pi^\infty)$
by tensoring by the identity on $\ell^2\N$ and further tensoring by the right regular representation of $\Gamma$ on $\ell^2\Gamma$. \\

\begin{theorem}\label{GammaCoarsly}\
 If $(Z', d')$ and $(Z, d)$ are {{$\Gamma$-proper  metric spaces}} with $\Gamma$-invariant fully supported Borel measures $\mu_{Z'}$ and $\mu_Z$ respectively,  such that $(Z', d')$ and $(Z, d)$ are $\Gamma$-coarsely equivalent. Then
  we have a group isomorphism:
 $$
 K_*\left(C^*_\Gamma (X; (Z', L^2Z'\otimes \ell^2\Gamma^\infty))\right) \xrightarrow{\cong} K_*\left(C^*_\Gamma (X; (Z, L^2Z\otimes \ell^2\Gamma^\infty))\right).
 $$
\end{theorem}

%
%
%

The proof of Theorem \ref{GammaCoarsly} will occupy the rest of this paragraph.
The notion of Roe covering $\Gamma$-isometry is introduced in Definition \ref{VRiso1}. According to Lemma \ref{Roeiso}, given a coarse $\Gamma$-map $f:Z'\to Z$ between   {{$\Gamma$-proper spaces}}, there always exist Roe covering $\Gamma$-isometries for $f$.

\begin{lemma}\label{functorial2}\cite{HRbook}
Given a Roe covering $\Gamma$-isometry $W: L^2Z'\otimes \ell^2\Gamma^\infty \longrightarrow L^2Z\otimes \ell^2\Gamma^\infty$ for the coarse $\Gamma$-map $f:Z'\to Z$,  the map
$
\Ad_W: C (X, \maL (L^2 Z'\otimes \ell^2\Gamma^\infty)_{*-str} )^\Gamma  \to C (X, \maL (L^2Z\otimes \ell^2\Gamma^\infty)_{*-str})^\Gamma$
 defined by $\Ad_W(T):=WTW^*$ induces a well-defined homomorphism
 $$
 \Ad_W: C^*_\Gamma (X; (Z', L^2Z'\otimes \ell^2\Gamma^\infty))  \longrightarrow C^*_\Gamma (X; (Z, L^2Z\otimes \ell^2\Gamma^\infty)).
 $$
 Moreover, the induced map $\Ad_{W, *}:K_*(C^*_\Gamma (X; (Z', L^2Z'\otimes \ell^2\Gamma^\infty)))\to K_*(C^*_\Gamma (X; (Z, L^2Z\otimes \ell^2\Gamma^\infty)))$
  is independent of the choice of the Roe covering $\Gamma$-isometry $W$ and will thus be denoted $f_*$.
 \end{lemma}

\begin{proof}\ The proof for $X=\{\bullet\}$ given in \cite{HRbook} extends immediately to our situation, and we recall the steps of this proof for the sake of completeness.
 The covering $\Gamma$-isometry is identified with {the isometry between the Hilbert $C(X)$-modules which is constant in the $X$-variable}. Notice first that if
 $T\in C^*_\Gamma (X; (Z', L^2Z'\otimes \ell^2\Gamma^\infty))$ has finite propagation, then so does $\Ad_W (T)$, precisely because $\Prop(W)$ is finite and $\Prop(WTW^*)\leq 2\Prop(W)+\Prop(T)$. Moreover, $\Ad_W(T)$ is
 locally compact as soon as $T$ is. Indeed,   if $\phi \in C_c(Y)$, then there exists $\phi'\in C_c(Y')$ such that $\pi^\infty_{Y}(\phi) W= \pi^\infty_Y(\phi)W\pi^\infty_{Y'}(\phi')$, so that since $T$ is locally compact, we have
 $$
 \pi^\infty_Y(\phi)WTW^*= \pi^\infty_Y(\phi)W(\pi^\infty_{Y'}(\phi')T)W^* \in C(X, \maK (L^2Z\otimes \ell^2\Gamma^\infty)).
 $$
Let us show now that two Roe covering $\Gamma$-isometries $W_1$ and $W_2$ induce the same map on $K$-theory.
The operators  $W_iW^*_j$ are clearly multipliers of the $C^*$-algebra  $C^*_\Gamma (X; (Z, L^2Z\otimes \ell^2\Gamma^\infty))$, for $i,j =1,2$. Let us show for instance  that $W_1W^*_2$ is such a multiplier. From the finite propagation of $W_1$ and $W_2$ we deduce that if $T\in C^*_\Gamma (X; (Z, L^2Z\otimes \ell^2\Gamma^\infty))$ has finite propagation,  then $W_1W_2^*T$ has finite propagation. To show that it is locally compact, we use again the fact that given $\phi\in C_c(Z)$, there exist $\phi'\in C_c(Z')$ and $\phi''\in C_c(Z)$ such that
$$
\pi^\infty_{Y}(\phi) W_1= \pi^\infty_Y(\phi)W_1\pi^\infty_{Y'}(\phi') \text{ and } \pi^\infty_{Y'}(\phi') W_2^*= \pi^\infty_{Y'} (\phi')W_2^* \pi^\infty_{Y}(\phi'').
$$
Hence for any $T\in C^*_\Gamma (X; (Z, L^2Z\otimes \ell^2\Gamma^\infty))$, we have
$$
\pi^\infty_Y(\phi)W_1W_2^*T= \pi^\infty_Y(\phi)W_1\pi^\infty_{Y'}(\phi')W_2^*T= \pi^\infty_Y(\phi)W_1\pi^\infty_{Y'}(\phi')W_2^*\pi^\infty_Y(\phi'')T.
$$
The latter operator is hence  compact since $\pi^\infty_Y(\phi'')T$ is compact.

Now it is easy to check that for $T\in C^*_\Gamma (X; (Z', L^2Z'\otimes \ell^2\Gamma^\infty))$,  the following relation holds
$$
\begin{bmatrix} W_1TW_1^* & 0\\ 0 & 0 \end{bmatrix} \sim  \begin{bmatrix} 0 & 0\\ 0 & W_2TW_2^* \end{bmatrix} \quad \text{ in } M_2(C^*_\Gamma (X; (Z, L^2Z\otimes \ell^2\Gamma^\infty))
$$
through conjugation by the unitary matrix $U= \begin{bmatrix} I- W_1W_1^* & W_1W_2^*\\ W_2W_1^* & I- W_2W_2^* \end{bmatrix}$, so that $(\Ad_{W_1})_* = (\Ad_{W_2})_*$ on $K$-theory.
\end{proof}

We thus end up, for any coarse $\Gamma$-map $f:Z'\to Z$, with the well-defined group morphism
$$
f_*: K_*(C^*_\Gamma (X; (Z', L^2Z'\otimes \ell^2\Gamma^\infty)))\longrightarrow K_*(C^*_\Gamma (X; (Z, L^2Z\otimes \ell^2\Gamma^\infty)))
$$
This construction is then clearly  functorial in the sense that if $f':Z''\to Z'$ is another coarse $\Gamma$-map, then
$$
(f\circ f')_* = f_*\circ f'_*.
$$

\begin{lemma}\label{functorial3}
  If $f,g: Z'\rightarrow Z$ are coarsely equivalent coarse $\Gamma$-maps, then
  $$
  f_*= g_* : K_*(C^*_\Gamma (X; (Z', L^2Z'\otimes \ell^2\Gamma^\infty)))\longrightarrow K_*(C^*_\Gamma (X; (Z, L^2Z\otimes \ell^2\Gamma^\infty))).
  $$
 \end{lemma}

 \begin{proof}
 A Roe covering $\Gamma$-isometry for $f$ is also a Roe covering $\Gamma$-isometry for $g$ and vice versa. Indeed,
 $d_Z( f(Z'), g(Z')) \leq M$ for some constant $M$, therefore, if $W$ is a $\Gamma$-equivariant isometry which has finite propagation with respect to $f$ then it automatically has finite propagation with respect to $g$. Since the maps $f_*$ and $g_*$ don't depend on the choice of the Roe covering $\Gamma$-isometry, the proof is complete.
 \end{proof}

{ Theorem \ref{GammaCoarsly} is now a corollary of Theorem \ref{Roeiso}, Lemma \ref{functorial2} and Lemma \ref{functorial3}}.
%

{{Recall that a metrically proper coarse map  $f: (Z', d_{Z'}) \to (Z, d_Z)$ is called a coarse embedding.  The map $f$ is coarsely onto if there exists a constant $C\geq 0$ such that
$$
\forall z\in Z, \exists z'\in Z'\text{ with } d_Z (f(z'), z) \leq C.
$$
A $\Gamma$-coarse embedding will be for us a coarse embedding which is  coarsely equivariant. }}

{{\begin{lemma}
Let $f: (Z', d_{Z'}) \to (Z, d_Z)$ be a $\Gamma$-coarse embedding which is coarsely onto. Then $f$ is $\Gamma$-coarse equivalence.
\end{lemma}}}

\begin{proof}
{{According to \cite{G14}, we know that there exists a Borel map  $h: (Z, d_{Z}) \to (Z', d_{Z'})$ which is a coarse embedding and which is a coarse inverse to $f$. Let $C$ be the onto-constant of $f$, i.e.  such that
$$
\forall z\in Z, \exists z'\in Z'\text{ with } d_Z (f(z'), z) \leq C.
$$
Then the map $h$ is defined in  \cite{G14} in such a way that it satisfies the relation
$$
d_Z (f(h(z)), z) \leq C, \quad \forall z\in Z.
$$
Now we have for any $z\in Z$ and any $g\in \Gamma$:
\begin{eqnarray*}
d_Z\left(f (h(gz)), f(gh(z))\right) & \leq  &  d_Z \left(f(h(gz)), gz\right) + d_Z\left(z, f(h(z))\right) + d_Z \left(g f(h(z)), f(gh(z))\right) \\
& \leq & C + C + M_1
\end{eqnarray*}
where $M_1$ is the coarse-equivariance constant for $f$, i.e. a constant which satisfies
$$
d_Z \left(g f(z'), f(g z') \right) \leq M_1,\quad \forall (z', g)\in Z'\times\Gamma.
$$
Since $f$ is metrically proper, we conclude from the previous estimate that there exists a constant $C'\geq 0$ such that
$$
d_{Z'} \left( h(gz), gh(z)\right) \leq C',
$$
and hence $h$ is coarsely equivariant as allowed.}}
\end{proof}

We  can finally deduce the following important corollary

\begin{corollary}\label{inclusioninv}
{{If the $\Gamma$-space $Z$ is  proper and $i: Z'\hookrightarrow Z$ is a coarse inclusion of a  closed  $\Gamma$-subspace $Z'$, then we have a group isomorphism:
$$
i_{Z'\subset Z}^C : K_*(C^*_\Gamma (X; (Z', L^2Z'\otimes \ell^2\Gamma^\infty)))\stackrel{\cong}{\longrightarrow} K_*(C^*_\Gamma (X; (Z, L^2Z\otimes \ell^2\Gamma^\infty)))
$$
In particular, when $Z/\Gamma$ is compact, any closed $\Gamma$-subspace $Z'$ induces  the above isomorphism $i_{Z'\subset Z}^C$.}}
\end{corollary}

\begin{proof}\
{{Indeed, in this case, the inclusion $i$ is automatically a $\Gamma$-coarse embedding which is coarsely onto. Hence the inclusion map $i$ a $\Gamma$-coarse equivalence according to the previous lemma.
The proof is then completed by applying Theorem \ref{GammaCoarsly} above.}}
\end{proof}

\subsection{Functoriality of $D^*$-algebras}\
 In order to prove a similar functoriality result for the $D^*$-algebras, we use the results of \cite{HRbook}[Chapter 12] and we shall need a  generalization of Voiculescu's theorem.  Recall the notion of {Roe-Voiculescu} covering $\Gamma$-isometry, see Definition \ref{VRiso2}.

\begin{lemma} \label{functorial4} Let $W$ be a {Roe-Voiculescu} covering $\Gamma$-isometry for the continuous {{$\Gamma$-coarse}} map $f: Z'\rightarrow Z$.
 The map $ \Ad_W: C (X; {\maL}(\ell^2\Gamma^\infty \otimes L^2Z')_{*-str})^{{\Gamma}} \longrightarrow C (X; {\maL}(\ell^2\Gamma^\infty \otimes L^2Z)_{*-str})^{{\Gamma}}$ given by $\Ad_W(T):=(\id_{C(X)} \otimes W)T(\id_{C(X)}\otimes W^*)$,
yields a well-defined $C^*$-algebra homomorphism
$$
\Ad_W: D^*_\Gamma (X; (Z', \ell^2\Gamma^\infty\otimes L^2Z')) \longrightarrow D^*_\Gamma (X; (Z, \ell^2\Gamma^\infty\otimes L^2Z)).
$$
Moreover, if $W_1$ and $W_2$ are two {Roe-Voiculescu} covering $\Gamma$-isometries for $f$, then they induce the same map on $K$-theory.
\end{lemma}

\begin{proof}\
 The arguments in \cite{SiegelThesis}[Propositions 3.3.12-3.3.15] can be adapted to our situation to prove this lemma.
Set $Y=X\times Z$, $Y'=X\times Z'$ and  let $E:=C(X)\otimes L^2Z'\otimes \ell^2\Gamma^\infty$ and $E':=C(X)\otimes L^2Z'\otimes \ell^2\Gamma^\infty$ be the corresponding Hilbert $\Gamma$-equivariant $C(X)$-modules. Recall the representations
\begin{multline*}
\pi_Y: C_0(Y)\simeq C(X)\otimes C_0(Z) \longrightarrow \maL_{C(X)} (C(X)\otimes L^2Z) \\ \text{ and }\pi_{Y'}: C_0(Y')\simeq C(X)\otimes C_0(Z') \longrightarrow \maL_{C(X)} (C(X)\otimes L^2Z').
\end{multline*}
We then set $\pi := \pi_Y\otimes \id_{\ell^2\Gamma^\infty}$  and $\pi':= \pi_{Y'}\otimes \id_{\ell^2\Gamma^\infty}$.
 The operator $P=\id_{C(X)}\otimes WW^*$ is an adjointable $\Gamma$-invariant projection, and we thus have a decomposition:
  $$
  E= PE \oplus (I-P)E
  $$
The representation $\pi$ then writes in  this decomposition
  $$
  \pi = \begin{bmatrix} \pi_{11} & \pi_{12}\\ \pi_{21} & \pi_{22}  \end{bmatrix},
  $$
with each diagonal element $\pi_{jj}$ for $j=1,2$, being a $*$-homomorphism  modulo $\maK_{C(X)}(E)$ and with the off-diagonal operators $\pi_{12}(\phi)$ and $\pi_{21}(\phi)$ being compact operators, for any $\phi\in C_0(Y)$. Indeed, let us use the standard notation for two adjointable operators $S_1, S_2 \in \maL_{C_0(X)}(\bullet, \bullet)$, we write $S_1\sim S_2$  if $S_1-S_2$ is a compact operator. Then given $\varphi_1, \varphi_2\in C_0(Y)$, we have
\begin{eqnarray*}
\pi_{11} (\varphi_1\varphi_2) - \pi_{11} (\varphi_1) \pi_{11} (\varphi_2) & = & P \pi(\varphi_1) (I-P) \pi(\varphi_2) P\\
 \sim  P \pi(\varphi_1) (\id_{C(X)}\otimes W) \pi'(\varphi_2\circ f) (\id_{C(X)}\otimes W^*)& - & P\pi(\varphi_1)  (\id_{C(X)}\otimes W) \pi'(\varphi_2\circ f) (\id_{C(X)}\otimes W^*) P\\
 \sim P (\id_{C(X)}\otimes W) \pi'(\varphi_1\circ f) \pi' (\varphi_2\circ f) (\id_{C(X)}\otimes W^*) & - &  (\id_{C(X)}\otimes W) \pi'(\varphi_1\circ f) \pi' (\varphi_2\circ f) (\id_{C(X)}\otimes W^*) P\\
 = (\id_{C(X)}\otimes W) \pi'(\varphi_1\circ f) \pi' (\varphi_2\circ f) (\id_{C(X)}\otimes W^*)  & - & (\id_{C(X)}\otimes W) \pi'(\varphi_1\circ f) \pi' (\varphi_2\circ f) (\id_{C(X)}\otimes W^*)P\\
 &=& 0
\end{eqnarray*}
{A} similar computation proves the other claims regarding $\pi_{22}$, $\pi_{12}$ and $\pi_{21}$.
  Then, using the following equations for any $\phi\in C_0(Y)$:
  $$
  \pi_{21}(\phi)= \pi (\phi)P- P\pi (\phi)P \quad \text{ and } \quad \pi_{12}(\phi)= P\pi (\phi)-P\pi (\phi)P,
  $$
we deduce that $[P,\pi (\phi)]\in \maK_{C_0(X)}(E)$.

Moreover, since $W$ has finite propagation by property (1), we deduce that $P \in D^*_\Gamma (X; (Z, \ell^2\Gamma^\infty\otimes L^2Z))$.

Let now $T$ be a given element of $D^*_\Gamma (X; (Z', \ell^2\Gamma^\infty\otimes L^2Z'))$.
Then  we have the following:
  \begin{eqnarray*}
   &&\pi'(\phi)T  \sim  T\pi'(\phi)\\
   &\Rightarrow& (\id_{C(X)}\otimes W^*)\pi(\phi)(\id_{C(X)} \otimes W)T\sim T(\id_{C(X)} \otimes W^*) \pi(\phi) (\id_{C(X)} \otimes W)\\
   &\Rightarrow& (\id_{C(X)} \otimes WW^*)\pi (\phi)(\id_{C(X)} \otimes W)T(\id_{C(X)} \otimes W^*)\sim (\id_{C(X)} \otimes W)T(\id_{C(X)} \otimes W^*)\pi(\phi)(\id_{C(X)} \otimes WW^*)\\
   &\Rightarrow& \pi(\phi)(\id_{C(X)} \otimes WW^*W)T(\id_{C(X)} \otimes W^*)\sim (\id_{C(X)} \otimes W)T(\id_{C(X)} \otimes W^*WW^*)\pi(\phi)\\
   &\Rightarrow& \pi (\phi)(\id_{C(X)} \otimes W)T(\id_{C(X)} \otimes W^*)\sim (\id_{C(X)} \otimes W)T(\id_{C(X)} \otimes W^*)\pi(\phi)\\
   &\Rightarrow& [\pi(\phi),(\id_{C(X)} \otimes W)T(\id_{C(X)} \otimes W^*)]\sim 0
  \end{eqnarray*}
Also, we have $\Prop((\id_{C(X)} \otimes W)T(\id_{C(X)} \otimes W^*)) \leq 2\Prop(W) + \Prop(T)$, so we get
$$
\Ad_W(T)\in D^*_\Gamma (X; (Z, \ell^2\Gamma^\infty\otimes L^2Z)).
$$
Let $W_1$ and $W_2$ be two {Roe-Voiculescu} covering $\Gamma$-isometries for $f$. Note that
$$
\id_{C(X)} \otimes W_jW^*_j \in D^*_\Gamma (X; (Z, \ell^2\Gamma^\infty\otimes L^2Z))\text{ for }j =1,2,
$$
 {as we have already shown}. We claim that $\id_{C(X)} \otimes W_1W^*_2$ and $\id_{C(X)} \otimes W_2W^*_1$ also belong to
$D^*_\Gamma (X; (Z, \ell^2\Gamma^\infty\otimes L^2Z))$. Given this claim, it is easy to check that for $T\in D^*_\Gamma (X; (Z', \ell^2\Gamma^\infty\otimes L^2Z'))$, the following relation
holds in $M_2(D^*_\Gamma (X; (Z, \ell^2\Gamma^\infty\otimes L^2Z)))$:
\begin{multline*}
(\id_{C(X)} \otimes U) \begin{bmatrix} (\id_{C(X)} \otimes W_1)T(\id_{C(X)} \otimes W_1^*) & 0\\ 0 & 0 \end{bmatrix} (\id_{C(X)} \otimes  U) =  \begin{bmatrix} 0 & 0\\ 0 & (\id_{C(X)} \otimes W_2)T(\id_{C(X)} \otimes W_2^*) \end{bmatrix},
\end{multline*}
where $U$ is the self-adjoint unitary matrix
$$
U= \begin{bmatrix} \id_{\ell^2\Gamma^\infty\otimes L^2Z}- W_1W_1^* & W_1W_2^*\\ W_2W_1^* & \id_{\ell^2\Gamma^\infty\otimes L^2Z}- W_2W_2^* \end{bmatrix}.
$$
Hence $(\Ad_{W_1})_* = (\Ad_{W_2})_*$ on $K$-theory.

Let us show the claim that $\id_{C(X)}\otimes W_1W^*_2 \in D^*_\Gamma (X; (Z, \ell^2\Gamma^\infty\otimes L^2Z))$ (the proof for $W_2W_1^*$ follows by symmetry). Since $(\id_{C(X)}\otimes W_1^*)\pi (\phi) (\id_{C(X)}\otimes W_1) \sim \pi'(\phi\circ f) \sim (\id_{C(X)}\otimes W_2^*) \pi (\phi) (\id_{C(X)}\otimes W_2)$, we have:
$$
(\id_{C(X)}\otimes W_1W_1^*) \pi (\phi) (\id_{C(X)}\otimes W_1W_2^*) \sim (\id_{C(X)}\otimes W_1W_2^*) \pi (\phi)(\id_{C(X)}\otimes W_2W_2^*).
$$
Since  $W_jW^*_j \in D^*_\Gamma (X; (Z, \ell^2\Gamma^\infty\otimes L^2Z))$ for $j =1,2$,  we get
$$
\pi (\phi)(\id_{C(X)}\otimes  W_1W_1^*W_1W_2^*) \sim (\id_{C(X)}\otimes W_1W_2^*W_2W_2^*) \pi (\phi) \Rightarrow [\pi (\phi), \id_{C(X)}\otimes W_1W_2^*] \sim 0.
$$
Again, $\Prop(W_1W_2^*)\leq \Prop(W_1)+\Prop(W_2)$, so in fact $W_1W_2^* \in D^*_\Gamma (X; (Z, \ell^2\Gamma^\infty\otimes L^2Z))$ and we are done.
\end{proof}

As a corollary of Lemma \ref{VR-isometry} and  Lemma \ref{functorial4}, we deduce

\begin{proposition}
Any continuous {{$\Gamma$-coarse map $f: Z'\to Z$}} induces a well defined group morphism
$$
f_*\;:\; K_*\left( D^*_\Gamma (X; (Z', \ell^2\Gamma^\infty\otimes L^2Z'))\right) \longrightarrow K_*\left( D^*_\Gamma (X; (Z, \ell^2\Gamma^\infty\otimes L^2Z))\right),
$$
as well as a well defined group morphism
$$
f_*\;:\; K_*\left( Q^*_\Gamma (X; (Z', \ell^2\Gamma^\infty\otimes L^2Z'))\right) \longrightarrow K_*\left( Q^*_\Gamma (X; (Z, \ell^2\Gamma^\infty\otimes L^2Z))\right),
$$
\end{proposition}

\begin{corollary}
If {{$i: Z'\hookrightarrow Z$ is a coarse inclusion of a closed $\Gamma$-subspace}}, then we have well defined induced group morphisms:
$$
i_{Z'\subset Z}^D : K_*(D^*_\Gamma (X; (Z', L^2Z'\otimes \ell^2\Gamma^\infty)))\stackrel{}{\longrightarrow} K_*(D^*_\Gamma (X; (Z, L^2Z\otimes \ell^2\Gamma^\infty)))
$$
and
$$
i_{Z'\subset Z}^Q : K_*(Q^*_\Gamma (X; (Z', L^2Z'\otimes \ell^2\Gamma^\infty)))\stackrel{}{\longrightarrow} K_*(Q^*_\Gamma (X; (Z, L^2Z\otimes \ell^2\Gamma^\infty))).
$$
{{In particular, if $Z$ is cocompact and $i: Z'\hookrightarrow Z$ is the inclusion of a closed $\Gamma$-subspace, then the group morphisms $i_{Z'\subset Z}^D$ and $i_{Z'\subset Z}^Q$ are well defined.}}
\end{corollary}

\begin{proof}
As already observed, the inclusion $i: Z'\hookrightarrow Z$ is  obviously continuous. Therefore, the corollary follows from Lemma \ref{functorial4}.
\end{proof}

\section{Paschke-Higson duality}\label{AppendixC}

We devote this section to the proof of the Paschke-Higson duality theorem for $\Gamma$-families. The classical version of this duality theorem can be consulted for instance in \cite{HigsonPaschke, Paschke}. {{In view of our main interest in this paper, namely the universal Higson-Roe sequence, we shall assume our proper $\Gamma$-spaces to be cocompact.}}

\subsection{Statement of the theorem}

We shall need some classical results due to Pimsner-Popa-Voiculescu \cite{PPV} which in turn extend the classical theorem of Voiculescu.
The goal of this section  is the proof of  the following  Paschke-Higson duality

\begin{theorem} \label{Paschkev2}
With the above notations {{and assuming that the action of $\Gamma$ on $Z$ is proper and cocompact}}, the Paschke map gives group isomorphisms
$$
K_i(Q^*_\Gamma(X; (Z, \ell^2\Gamma^\infty\otimes L^2Z)) \xrightarrow{\maP} KK_{\Gamma}^{i+1}(Z, X), \quad i\in \Z_2.
$$

\end{theorem}

This theorem is already interesting when $Z$ is compact and $\Gamma$ is trivial. In this case, we get the following result which is a rephrasing of classical results from \cite{PPV}:

\begin{theorem}[Paschke duality theorem, the non-equivariant case]\label{Paschkev1}\
For any fiberwise ample  $C(X)$-representation of $C(X\times Z)$ in the Hilbert module $C(X)\otimes H$, we have  group isomorphisms
$$
K_i(Q^* (X; (Z,H))) \stackrel{\maP}{\longrightarrow} KK^{i+1} (Z, X), \quad i=0, 1.
$$
\end{theorem}

Let us explain the relation with the PPV work in this non-equivariant case. Set  $p_2:X\times Z\to Z$ for the second projection. A $C(X)$-representation $\what\pi: C(X\times Z)\rightarrow \maL_{C(X)}(\maE)$ on a Hilbert $C(X)$-module $\maE$ corresponds to a family of representations $\pi_x: C(Z)\rightarrow \maL(\maE_x)$ on the associated field of Hilbert spaces, given by localizing at $x$. Notice that one can use
the homomorphism $p_2^*$ to deduce a representation $\pi: C(Z)\rightarrow \maL_{C(X)}(\maE)$, which is the one associated with this field $(\pi_x)_{x\in X}$ of Hilbert space representations.

In the terminology of the seminal paper \cite{PPV}[Sections 1 and 2] which considers the case of the free module $\maE= C(X)\otimes H$,  {the $C(X)$-representation $\what\pi: C(X\times Z) \rightarrow \maL_{C(X)}(\maE)$
 is fiberwise ample}, that is the corresponding field of representations {is} composed of ample representations (see more precisely Definition \ref{FiberwiseAmple} below),
if and only if the {$X$-extension} associated to the representation $\pi: C(Z)\rightarrow \maL_{C(X)}(\maE)$ has trivial ideal symbol, i.e. the $X$-extension that $\pi$ induces is homogeneous. If for instance $\chi$ is a given ample representation of $C(Z)$ in the Hilbert space $H$, then the associated $C(X)$-representation $\what\chi$ of $C(X\times Z)$ in $\maL_{C(X)} (C(X)\otimes H)$ is clearly fiberwise ample.
Any fiberwise ample representation $\what\pi: C(X\times Z)\rightarrow \maL_{C(X)} (C(X)\otimes H)$ gives rise to a trivial $X$-extension through its associate representation $\pi$, as follows. Let $B:= \pi (C(Z))+ C(X)\otimes \maK(H)\subseteq \maL_{C(X)} (C(X)\otimes H)$ and if $p$ is the Calkin projection for the Hilbert module $C(X)\otimes H$, then we also set $\tau:= p \circ \pi :C(Z)\rightarrow \maL_{C(X)}(C(X) \otimes H)/C(X)\otimes \maK(H)$ which is then a monomorphism. This yields the trivial $X$-extension in the terminology of \cite{PPV}:
$$
0\rightarrow C(X)\otimes \maK(H) \hookrightarrow B\xrightarrow{\sigma} C(Z) \rightarrow 0
$$
with $\sigma$ given by the composite map $B\rightarrow B /C(X)\otimes \maK(H)\xrightarrow{\cong} \tau (C( Z))\xrightarrow{\cong} C(Z)$. As mentioned above this $X$-extension is in fact homogeneous.

Then, we can rewrite the main theorem from \cite{PPV} that is used here:

\begin{proposition}[\cite{PPV}]\label{VRisometry} Let $\what\pi_1$ and $\what\pi_2$ be two fiberwise ample $C(X)$-representations of $A=C(X\times Z)$ in the Hilbert modules $C(X)\otimes H_1$  and $C(X)\otimes H_2$ respectively. Then there exists a unitary $S\in \maL_{C(X)} (C(X)\otimes H_1, C(X)\otimes H_2)$ such that
  $$
  S^*\what\pi_2 (\varphi) S-\what\pi_1(\varphi) \in C(X)\otimes \maK(H_1) \text{ for all } \varphi\in C(X\times Z).
  $$
 \end{proposition}

\begin{proof}
This proposition is a corollary of the more general statement proved in \cite{PPV}[Proposition 2.9]. More precisely,
the two $C(X)$-representations $\what\pi_1$ and $\what\pi_2$  yield the representations $\pi_1$ and $\pi_2$ of $C(Z)$,
associated as above by composing with $p_2^*$. Since the ideal symbols of $\pi_1$ and $\pi_2$ are homogeneous, the PPV theorem insures the existence of a unitary $S\in \maL_{C(X)} (C(X)\otimes H_1, C(X)\otimes H_2)$ such that
$$
  S^*\pi_2 (f) S-\pi_1(f) \in C(X)\otimes \maK(H_1) \text{ for all } f \in C(Z).
$$
But then for any given continuous function $u$ on $X$, we have by definition of $C(X)$-representations:
$$
S^*\what\pi_2 (u\otimes f) S-\what\pi_1(u\otimes f) = S^*\circ [R_u\circ \what\pi_2 (1\otimes f)] \circ  S - R_u\circ \what\pi_1 (1\otimes f),
$$
where $R_u$ is multiplication on the right by $u$, say the module structure of $C(X)\otimes H_1$. So, $S^*$ being $C(X)$ linear, we get
$$
S^*\what\pi_2 (u\otimes f) S-\what\pi_1(u\otimes f) = R_u\circ [S^*\pi_2 (f) S-\pi_1(f)]
$$
Since $S^*\pi_2 (f) S-\pi_1(f)$ is a compact operator of the Hilbert module $C(X)\otimes H_1$, the proof is complete.
\end{proof}

Theorem \ref{Paschkev1} is then an easy corollary of the previous proposition. Indeed, let us treat the case $i=0$ since the argument is similar  for $i=1$.
Given a projection $P$ in $Q^*(X; (Z,H))$, the triple $(\pi, C(X) \otimes H, 2P-I)$ is a Kasparov cycle for the pair of $C^*$-algebras $(C(Z), C(X))$. Moreover, it is easy to check that  this yields a well defined group morphism
$$
K_0(Q^*(X; (Z,H))) \stackrel{\maP}{\longrightarrow} KK^{1}(Z, X)
$$
obtained  by setting $\maP([P])=[(\pi, C(X) \otimes H, 2P-I)]$. An inverse for the map $\maP$ is then constructed as follows. Let $y \in KK^{1}(Z, X)$ be represented by a Kasparov cycle $(\pi', E', F')$. Thanks to the Kasparov stabilisation
theorem, we can assume without lost of generality that $E'$ is a submodule of some $C(X)\otimes H'$. Replacing $\pi'$ by $\pi'\oplus \pi$ and {$F'$ by $\diag (F',\id)$}, one can also assume that $\pi'$ is fibrewise ample. So, by {Proposition} [\ref{VRisometry}], there is a unitary
$S \in \maL_{C(X)}(C(X) \otimes H', C(X) \otimes H)$ such that $S^*\pi(f)S- \pi'(f) \in \maK_{C(X)} (C(X)\otimes H')$ for any $f\in C(Z)$. Then the cycle $(\pi, C(X) \otimes H, SF'S^*)$ is equivalent to the cycle $(\pi',E',F')$. Setting
$$
\tilde{P'}= \frac{S(1+{F'})S^*}{2},
$$
it is not difficult to check using the conditions on a $KK$-cycle, 
that $\tilde{P'}$ is a projection in $Q^*(X; (Z,H))$.
The map $\maP': KK^1(Z, X)\rightarrow K_0(Q^*(X;(Z,H)))$ given by
$$
\maP'([\pi',E',F')]= [\tilde{P'}]
$$
 can then be verified to be a well-defined group homorphism, and to be an inverse to $\maP$.\\

Our strategy to prove  Theorem \ref{Paschkev2} will be, no surprise,  to use  an extended version  of the Pimsner-Popa-Voiculescu theorem.

\subsection{Proof of the Paschke-Higson theorem}

Consider  the finite dimensional compact metrizable space $X$ and the locally compact metric space $Z$. Finite dimension is needed inorder to apply the results of \cite{PPV} which used the Michael selection theorem.   We consider the proper second projection $p_2:X\times Z\to Z$ so that any $C(X)$-representation $\what\pi: C_0(X\times Z) \rightarrow \maL_{C(X)} (\maE)$ in the Hilbert $C(X)$-module $\maE$ induces using $p_2^*$, a representation  $\pi:C_0(Z)\rightarrow \maL_{C(X)} (\maE)$ which is  associated with the field $\pi_{x}: C_0(Z)\rightarrow \maL (\maE_x)$  of Hilbert spaces representations in the associated field of Hilbert spaces, obtained by localizing at any given $x\in X$.

\begin{definition}\label{FiberwiseAmple} The $C(X)$-representation $\what\pi: C_0(X\times Z)\rightarrow \maL_{C(X)} (\maE)$ will be called a fibrewise ample representation if for any $x\in X$, the representation $\pi_{x}: C_0(Z)\rightarrow \maL (\maE_x)$ is ample, i.e.  for any $x\in X$, $\pi_x$ is non-degenerate and one has
$$
\pi_{x} (f) \in \maK(\maE_x) \Longrightarrow f=0, \quad  \text{ for any }f\in C_0(Z).
$$
\end{definition}


The right regular  representation of $\Gamma$  in the Hilbert space $\ell^2\Gamma$ is denoted $\rho$, and its tensored product by the identity of $\ell^2\N$ is the unitary representation $\rho^\infty$ of $\Gamma$ in $\ell^2\Gamma^\infty$. We shall use  the following generalization of the PPV theorem, whose detailed proof is tedious and is expanded in \cite{ExtendedPPV}.

\medskip

 \begin{theorem}\label{EquivPPV}
Let $\what\pi_1$ and $\what\pi_2$ be two fiberwise ample $\Gamma$-equivariant representations of $A=C_0(X\times Z)$ in the Hilbert $\Gamma$-equivariant $C(X)$-modules $C(X)\otimes H_1$  and $C(X)\otimes H_2$ respectively. Then, identifying each $\what\pi_i$ with the trivially extended representation $\left(\begin{array}{cc} \what\pi_i & 0\\ 0 & 0\end{array}\right)$ that is further tensored by the identity of $\ell^2\Gamma^\infty$, there exists a $\Gamma$-invariant unitary operator with finite propagation
$$
W\in \maL_{C(X)} (C(X)\otimes {(H_1\oplus H_2)\otimes \ell^2\Gamma^\infty, C(X)\otimes (H_2\oplus H_1)\otimes \ell^2\Gamma^\infty}),
$$
 such that
$$
  W^*\what\pi_2 (\varphi) W - \what\pi_1(\varphi) \in C(X)\otimes \maK[{(H_1\oplus H_2)}\otimes \ell^2\Gamma^\infty], \quad \text{ for all } \varphi\in C_0(X\times Z).
  $$
\end{theorem}
%

\medskip

Using Theorem \ref{EquivPPV}, we can now give the details of the proof of the Paschke-Higson theorem.

%

\begin{proof}(of Theorem \ref{Paschkev2})

 We shall construct a group homomorphism $\maP': KK^1_{ \Gamma}(Z, X)\rightarrow K_0(Q^*_\Gamma(X, (Z, \ell^2\Gamma^\infty\otimes L^2Z))$, which will be an inverse for the
 Paschke-Higson morphism $\maP$. The argument is similar to \cite{GWY17} (Appendix A), except that we don't use Kasparov's generalization of Voiculescu theorem and rather apply the PPV result.\\

  \textbf{Step 1}: Let $[(\sigma, E, F)]\in KK^1_{ \Gamma}(Z, X)$. We may assume as usual that $\sigma$ is non-degenerate and that $F$ is self-adjoint. {{We may also replace  if necessary $L^2Z$ by its amplification $(L^2Z)^\infty$ if needed. We assume for simplicity that this amplification is not needed. We first proceed to some reductions in order to be in position to apply Theorem \ref{EquivPPV}.}}
  Adding a degenerate cycle of the form $[\what\pi_Y, L^2Z\otimes C(X), \id]$ and using the non-equivariant Kasparov stabilization
  theorem, we obtain a cycle of the form $[\sigma_1, L^2Z\otimes C(X), F_1]$, which endowed with the transported $\Gamma$-action, lies in the same $KK^1_\Gamma$-class as $[\sigma, E, F]$. More precisely, the representation $\sigma_1$ is the transport of the representation $\sigma\oplus \what\pi_Y$
  via conjugation with the unitary given by Kasparov stabilisation isomorphism $E\oplus L^2Z\otimes C(X)\cong L^2Z\otimes C(X)$. The representation is then fiberwise ample with our assumptions.  Note that the $\Gamma$-action in the new cycle, is also  taken to be the transport through the Kasparov isomorphism  of the $\Gamma$-action on $E\oplus (L^2Z\otimes C(X))$ (with the second component endowed with its usual $\Gamma$-action inherited from the action on $Z$). In particular this action $V$ may differ from the original $\Gamma$-action on $L^2Z\otimes C(X)$.  The operator $F_1$ is of course the transport of $F\oplus \id$ via the same Kasparov  isomorphism.\\

  \textbf{Step 2}: Embed $L^2Z\otimes C(X)$ equivariantly in $\ell^2\Gamma\otimes L^2Z\otimes C(X)$ via an equivariant isometry
  $S: L^2Z\otimes C(X)\rightarrow  \ell^2(\Gamma)\otimes L^2Z\otimes C(X)$, defined by the following formula which uses a cut-off function $c\in C_c(X\times Z)$:
 $$
 S(e)= \sum_{g\in \Gamma} \delta_{{{g^{-1}}}}\otimes \sigma_1(g\sqrt{c})(e) \quad \text{ for }e\in C_c(Z\times X),
 $$
  where the $\Gamma$-action on $L^2Z\otimes C(X)$ is given by the action $V$ from Step 1, while the $\Gamma$-action on $\ell^2(\Gamma)\otimes L^2Z\otimes C(X)$ is given by the
  right regular representation of $\Gamma$ on $\ell^2\Gamma$ tensored by the same action $V$. Notice that $S^*$ is induced by the formula $S^* (\delta_g\otimes e) = \sigma_1(g^{-1}{\sqrt{c}}) (e)$, and hence  we get that the projection $SS^*$ is induced by the formula
  $$
  {{SS^* (\delta_k\otimes e) := \sum_{g\in \Gamma} \delta_{g^{-1}} \otimes \sigma_1 \left({\sqrt{(g c) (k^{-1} c)}}\right) (e)}}.
  $$
 Now the Kasparov $\Gamma$-equivariant cycle
  $\left(S\sigma_1(\bullet)S^*, SS^*(\ell^2\Gamma\otimes L^2Z\otimes C(X)), SF_1S^*\right)$ represents the same $KK_\Gamma$-class as $[\sigma_1, L^2Z\otimes C(X), F_1]$. Note that
  since $S\sigma_1(\bullet)= (\id_{\ell^2\Gamma}\otimes \sigma_1(\bullet))S$, the projection $SS^*$ commutes with $\id_{\ell^2\Gamma}\otimes \sigma_1$ and we have the relation
  $$
  S\sigma_1(\bullet)S^*= SS^*(\id_{\ell^2\Gamma}\otimes \sigma_1(\bullet))SS^*
  $$
  where $\id_{\ell^2\Gamma}\otimes \sigma_1$ is viewed as a representation on $\ell^2\Gamma\otimes L^2Z\otimes C(X)$ as usual.  Hence the right hand side is a representation in the Hilbert module $SS^*(\ell^2\Gamma\otimes L^2Z\otimes C(X))$, indeed we have more precisely
  $$
  SS^*(\id_{\ell^2\Gamma}\otimes \sigma_1(\bullet)) = (\id_{\ell^2\Gamma}\otimes \sigma_1(\bullet)) SS^*.
  $$
  Adding a degenerate
  cycle of the form $[P'(\id_{\ell^2\Gamma}\otimes \sigma_1)(\bullet)P', P'(\ell^2\Gamma\otimes L^2Z\otimes C(X)), \id_{\Im P'}]$ where $P'$ is the projection
  $(\id_{\ell^2\Gamma\otimes L^2Z\otimes C(X)}- SS^*)$, we obtain
  a cycle $\left(\sigma_2:= \id_{\ell^2\Gamma}\otimes \sigma_1, \ell^2\Gamma\otimes L^2Z\otimes C(X), F_2:=(F_1\oplus \id)\right)$ which represents the same $KK_\Gamma$-class as the cycle obtained in Step 1. Notice that we have here
$$
P'(\id_{\ell^2\Gamma}\otimes \sigma_1)(\bullet)SS^* = 0\text{ and }SS^*(\id_{\ell^2\Gamma}\otimes \sigma_1)(\bullet)P'=0.
$$

  \textbf{Step 3}: Adding degenerate cycles to $[\sigma_2, \ell^2\Gamma\otimes L^2Z\otimes C(X), F_2]$ we may pass to a new $\Gamma$-equivariant Kasparov cycle
  $\left({{\sigma_2^\infty:=}}\id_{\ell^2\N}\otimes \sigma_2, \ell^2\Gamma^\infty \otimes L^2Z\otimes C(X), F_2^\infty:=\diag(F_2, \id, \id...)\right)$ which represents  the same $KK_\Gamma$-class. Let us further add  the degenerate cycle $\left(0,\ell^2\Gamma^\infty\otimes L^2Z\otimes C(X), 0\right)$ to $[\sigma_2^\infty,  \ell^2\Gamma^\infty\otimes L^2Z\otimes C(X), F^\infty_2]$ with the $\Gamma$-action
  now taken as the one coming canonically from the $\Gamma$-action on $X\times Z$ tensored with the right regular representation on the factor $\ell^2\Gamma$ and extended trivially on $\ell^2\N$. We obtain in this way  a new $\Gamma$-equivariant Kasparov cycle
  $$
  \left(\sigma_3:= \sigma^\infty_2\oplus 0,\ell^2\Gamma^\infty\otimes (L^2Z\oplus L^2Z)\otimes C(X), F_3:= F_2{{^\infty}}\oplus 0\right)
  $$
  still remaining in the same $KK_\Gamma$-class. Note that in the direct sum the first
  factor has a $\Gamma$-action coming from  Step 1 while the second factor carries the canonically induced action. \\

  \textbf{Step 4}: Since $\sigma_3$ is fibrewise ample and of the form $\id_{\ell^2\Gamma^\infty}\otimes \sigma_1 \oplus 0$, we are now in position to  apply Theorem \ref{EquivPPV}, so we  obtain a $\Gamma$-invariant unitary $W$ such that
  $$
  W \sigma_3(f)W^*- (\what\pi^\infty_Y(f)\oplus 0) \in \maK_{C(X)}\left(\ell^2\Gamma^\infty\otimes (L^2Z\oplus L^2Z)\otimes C(X)\right),\quad  \text{ for all } f\in C_0(X\times Z).
  $$
By Kasparov's homological equivalence Lemma \ref{KK_G-equiv}, the cycles
  $$
  \left(\sigma_3,\ell^2\Gamma^\infty\otimes (L^2Z\oplus L^2Z)\otimes C(X), F_3\right)\quad \text{ and }\quad
 \left(\what\pi^\infty_Y\oplus 0, \ell^2\Gamma^\infty\otimes (L^2Z\oplus L^2Z)\otimes C(X), F_4\right),
 $$
live in the same $KK_\Gamma$-class. Here of course $F_4:= WF_3W^*$. It is worth pointing out that:
  \begin{enumerate}
  \item the equivariant unitary $W$ interchanges the copies of $\ell^2\Gamma^\infty\otimes L^2Z\otimes C(X)$ in the direct
  sum with the two different $\Gamma$-actions as described in Step 3.
  \item if $W$ and $W'$ are two equivariant unitaries intertwining $\sigma_2\oplus0$ and $\what\pi^\infty_Y(f)\oplus 0$ up to compacts, then the unitary $W'W^*$
  intertwines $\what\pi^\infty_Y(f)\oplus 0$ with itself up to compacts, so another application of Kasparov's homological equivalence lemma \ref{KK_G-equiv} applied to
  $S=W'W^*$ and the representations $\pi_1=\pi_2=\what\pi^\infty_Y\oplus 0$ shows as well that the cycles
 \begin{multline*}
  \left(\what\pi^\infty_Y\oplus 0, \ell^2\Gamma^\infty\otimes (L^2Z\oplus L^2Z)\otimes C(X), WF_3W^*\right)\\ \text{  and }\left(\what\pi^\infty_Y\oplus 0, \ell^2\Gamma^\infty\otimes (L^2Z\oplus L^2Z)\otimes C(X),W'F_3(W')^*\right)
 \end{multline*}
  are in the same $KK_\Gamma$-class.
  \end{enumerate}

  \textbf{Step 5}:   Let $\tilde{F}_3$ be the $(1,1)$-entry in the matrix decomposition of $F_3$. Then
  the cycle
  $$
  [\what\pi^\infty_Y\oplus 0, \ell^2\Gamma^\infty\otimes (L^2Z\oplus L^2Z)\otimes C(X), F_4]
  $$
  is in the same $KK_\Gamma$-class as the
  cycle $[\what\pi^\infty_Y, \ell^2\Gamma^\infty\otimes L^2Z\otimes C(X), F_5:=W_{11}\tilde{F}_3W_{11}^*]$. Notice that the off-diagonal entries in the matrix decomposition of $F_4$ are locally compact and that the cycle corresponding to the $(2,2)$ element is degenerate.
Lastly,  we modify $F_5$ so that it is $\Gamma$-invariant
 by using the properness and cocompactness of the action. This is done by replacing as usual $F_5$ by $\Av(\hat{\pi}_Y^\infty(\sqrt{c})F_5\hat{\pi}_Y^\infty(\sqrt{c}))$ for a compactly supported continuous cut-off function $c$.
  The cycle for this latter operator induces the same $KK_\Gamma$-class since it is a locally compact perturbation of $F_5$. Note that the new $\Gamma$-equivariant replacement has finite propagation
  as well; the propagation being bounded above by $\diam(\supp(c))$. We  continue to denote by $F_5$ the new $\Gamma$-invariant operator, by abuse of notation.\\

  \textbf{Step 6}: We are now in position to define define the allowed inverse map $\maP':  {KK_{\Gamma}^{1}(Z, X)}\rightarrow K_0(Q^*_\Gamma(X; (Z, \ell^2\Gamma^\infty\otimes L^2Z))$ by setting
  $$
  \maP'([\sigma, E, F]):= \left[q\left(\frac{1}{2}(W_{11}W_{11}^* + F_5)\right)\right]
  $$
  where $q: D^*_\Gamma(X; (Z, \ell^2\Gamma^\infty\otimes L^2Z))\rightarrow Q^*_\Gamma(X; (Z, \ell^2\Gamma^\infty\otimes L^2Z))$ is the quotient projection.
 The operator $F_6:=q(\frac{1}{2}(W_{11}W_{11}^* + F_5))$ is then a projection in $Q^*_\Gamma(X; (Z, \ell^2\Gamma^\infty\otimes L^2Z))$, see  Lemma \ref{proj} below. Let us check now that $\maP'$ is well defined.
%
%

{Indeed}, if $F_t, t\in [0,1]$ {is} an operator homotopy between $KK_\Gamma$-cycles $(\sigma, E, F_0)$ and $(\sigma, E, F_1)$, then tracing the construction of the map $\maP'$ above, one easily deduces that the corresponding projections $F_6^0$ and $F_6^1$ are operator homotopic via $F_6^t$, $t\in [0,1]$.
 Suppose on the other hand that $(\sigma, E, F)$ is degenerate, then it is operator homotopic
  to the cycle $(\sigma, E, \id)$. Again tracing the construction of $F_6$ we see that it is given by $q(W_{11}W_{11}^*)$. Therefore, we have $\maP'([\sigma, E, F])= \maP'([\sigma, E, \id])$, while
  due to property (3) in the proof of Lemma \ref{proj} below we have that $q(W_{11}W_{11}^*)= q(\id)$. By a straightforward adaptation of the Eilenberg swindle argument in
  \cite{HRbook}, Proposition 8.2.8, it is easily seen that
  $$
  [\id]=0\in K_0(D^*_\Gamma(X; (Z, \ell^2\Gamma^\infty\otimes L^2Z))
  $$
  and thus $[F_6]=q_*[\id]=0 \in K_0(Q^*_\Gamma(X; (Z, \ell^2\Gamma^\infty\otimes L^2Z))$.
  Finally, by using Remark (2) in Step (4) above, one obtains invariance under unitary equivalence of cycles.  Thus $\maP'$ is well-defined and it is also straightforward to check that it is a group homomorphism.


 It is now clear, using property (3)  in the proof of Lemma \ref{proj} below,  that $\maP\maP'=\id$, so that $\maP$ is surjective.
 Indeed, we have by definition of $\maP$ that $ \maP(F_6)$ is represented by the cycle
 $$
\left(\pi_Y^\infty, \ell^2\Gamma^\infty\otimes L^2Z\otimes C(X), (F_5+W_{11}W_{11}^*)-\id\right).
 $$
 Note that since the operator $F_5+W_{11}W^*_{11}-\id$ is a locally compact perturbation of $F_5$ due to property (3) in the proof of Lemma \ref{proj}, the cycle
 $(\pi_Y^\infty, \ell^2\Gamma^\infty\otimes L^2Z\otimes C(X), F_5+W_{11}W_{11}^*-\id)$ is in the same $KK_\Gamma$-class as $(\pi_Y^\infty, \ell^2\Gamma^\infty\otimes L^2Z\otimes C(X), F_5)$, which by the constructions
 in Steps 1-5 above, is in the same class as the original cycle $(\sigma, E, F)$ that we started out with in Step 1. This shows that $\maP\maP'= \id$ on $KK_\Gamma^1(Z, X)$.

We now show by direct inspection
 that $\maP$ is injective. Indeed, if the image cycle $[\pi^\infty_Y, \ell^2\Gamma^\infty\otimes L^2Z\otimes C(X), 2P-\id]$ is degenerate, one shows that an Eilenberg swindle argument
 applies as follows. The operator $F:=2P-\id$ then satisfies the following relations:
 $$
 [F, \pi^\infty_Y(a)]=0 \text{ and } \pi^\infty_Y(a)(F^2-\id)=0,\quad  \forall a\in C(X).
 $$
In particular, the first relation implies that $[P, \pi^\infty_Y(a)]=0$ for all $a\in C(X)$, and the second relation implies that $P^2-P=0$, since $\pi_Y^\infty$ is a non-degenerate representation. Therefore $P\in \Proj(D^*_\Gamma(X; (Z, \ell^2\Gamma^\infty\otimes L^2Z))$ is a degenerate element and is susceptible to an Eilenberg swindle again as in \cite{HRbook}, Proposition 8.2.8. Therefore
the class $[P]=0 \in K_0(D^*_\Gamma(X; (Z, \ell^2\Gamma^\infty\otimes L^2Z))$, thus $q_*[P]=0 \in K_0(Q^*_\Gamma(X; (Z, \ell^2\Gamma^\infty\otimes L^2Z))$. In general, if there exists
an operator  homotopy between the class $(\pi^\infty_Y, \ell^2\Gamma^\infty\otimes L^2Z\otimes C(X), 2P-\id)$ and a degenerate cycle, the operator homotopy lifts to a homotopy of projections in
$Q^*_\Gamma(X; (Z, \ell^2\Gamma^\infty\otimes L^2Z)$, which connects the operator $P$ to a degenerate projection in $D^*_\Gamma(X; (Z, \ell^2\Gamma^\infty\otimes L^2Z))$
and therefore is zero in $K$-theory. Unitary equivalences and direct sums can also be handled similarly.
\end{proof}

We have used in the previous proof the following lemma:

\begin{lemma}\label{proj}
The operator  $F_6$ used above is a projection in the $C^*$-algebra $Q^*_\Gamma(X; (Z, \ell^2\Gamma^\infty\otimes L^2Z))$.
\end{lemma}

\begin{proof}
Notice that the matrix elements of the Voiculescu unitary $W$ satisfy the following properties for any $f\in C_0(X\times Z)$, denoting $w= W_{11}$:
\begin{enumerate}
\item $W_{ij}$ is $\Gamma$-invariant  and has finite propagation, for $1\leq i,j\leq 2$;
 \item $\sigma_2^\infty(f)- w^* \what\pi^\infty_Y(f) w \sim 0, w \sigma_2^\infty(f) w^*-  \what\pi^\infty_Y(f) \sim 0, W_{12}^*\sigma_2(f) \sim 0 \text{ and } \what\pi_Y^\infty(f)W_{12} \sim 0$;
\item $\what\pi_Y^\infty(f)(ww^*- \id) \sim 0 \text{ and } (w^*w- \id)\sigma_2(f)\sim 0$;
\item $[ww^*,\what\pi_Y^\infty(f)]\sim 0$;
\item $[wF_3 w^*, \what\pi_Y^\infty(f)]\sim 0$;
\item $\what\pi^\infty_Y(f)((wF_3w^*)^2- ww^*)\sim 0$.
\end{enumerate}
All these properties can be checked by straightforward verification. Let us check for instance property (6):
\begin{eqnarray*}
\what\pi^\infty_Y(f)((wF_3w^*)^2- ww^*) &\sim & wF_3w^*\what\pi_Y^\infty(f) w F_3w^* - \what\pi_Y^\infty(f) ww^*\\
&\sim & wF_3\sigma_2^\infty(f) F_3 w^* -\what\pi_Y^\infty(f) ww^*\\
&\sim & w\sigma_2^\infty(f) F_3^2 w^* -\what\pi_Y^\infty(f) ww^*\\
&\sim & w\sigma_2^\infty(f) w^* -\what\pi_Y^\infty(f) ww^*\\
&\sim & \what\pi_Y^\infty(f) -\what\pi_Y^\infty(f) ww^* \sim 0\\
\end{eqnarray*}
Since $\what\pi^\infty_Y(\bullet)(F_6^2- F_6)\sim \what\pi^\infty_Y(\bullet)((wF_3w^*)^2- ww^*)\sim 0$, by the finite propagation of
$F_6$ in Step 6 in the proof of Theorem \ref{Paschkev2} and since $w$ has finite propagation, we get the conclusion.

\end{proof}

\section{The universal HR sequence}

Our goal in this section is to provide the universal Higson-Roe sequence for these locally compact \'etale groupoids.
We denote by ${\underline E}\Gamma$ a locally compact Hausdorff model for the classifying space of proper $\Gamma$-actions. So, $\Gamma$ acts properly on ${\underline E}\Gamma$ with the usual contractibility condition, see \cite{TuHyper}. It is not true in general that the action of $\Gamma$ on ${\underline E}\Gamma$ is cocompact and we introduce the following definition.

\begin{definition}
We introduce the analytic surgery group $S_1(X\rtimes \Gamma)$ associated with the transformation groupoid $X\rtimes \Gamma$ as:
$$
S_{*+1}(X\rtimes \Gamma):= \lim_{\stackrel{\longrightarrow}{Z\subset \underline{E}\Gamma}} K_{*}\left(D^*_\Gamma (X ; (Z, L^2Z\otimes \ell^2\Gamma^\infty))\right), \text{ for }*=0, 1\in \Z_2,
$$
where the direct limit is taken with respect to inclusion of $\Gamma$-invariant, $\Gamma$-compact closed subspaces $Z\subseteq \underline{E}\Gamma$, and using the system of group morphisms
$$
i_{Z'\subset Z}^D : K_*(D^*_\Gamma (X; (Z', L^2Z'\otimes \ell^2\Gamma^\infty)))\stackrel{}{\longrightarrow} K_*(D^*_\Gamma (X; (Z, L^2Z\otimes \ell^2\Gamma^\infty)))
$$
associated with inclusions of cocompact closed subspaces of $\underline{E}\Gamma$.
\end{definition}
Similar to the case of countable groups, the groups $S_*(X\rtimes \Gamma)$ can be interpreted as defect groups and will  enter in a long exact sequence involving the Baum-Connes maps, see \cite{HR-I, HR-II, HR-III} and also \cite{BenameurRoyI}.

Recall on the other hand that the LHS group in the Baum-Connes morphism for the groupoid $X\rtimes \Gamma$ can be  described as follows (see  \cite{BaumConnesHigson}):
$$
RK_* (X\rtimes \Gamma) \; := \;   \lim_{\stackrel{\longrightarrow}{Z\subset \underline{E}\Gamma}} KK_\Gamma^* (Z, X),
$$
where the limit is again taken with respect to the inductive system associated with the inclusions $i:Z'\hookrightarrow Z$ of $\Gamma$-compact closed subspaces of $\underline{E}\Gamma$ and using the induced functoriality morphisms $i_*: KK_\Gamma^* (Z', X)\rightarrow KK_\Gamma^* (Z, X)$.
We are now in position to state the main theorem of this section.

\medskip

\begin{theorem}\label{Surgeryseq} There exists a six-term exact sequence in $K$-theory:
$$
...\rightarrow RK_0(X\rtimes\Gamma)\xrightarrow{\mu^{BC}_0} K_0(C(X)\rtimes_{\redg}\Gamma)\rightarrow S_1(X\rtimes\Gamma)\rightarrow RK_1(X\rtimes\Gamma)\xrightarrow{\mu_1^{BC}} K_1(C(X)\rtimes_{\redg}\Gamma)\rightarrow...
$$
where $\mu^{BC}_*$ is the Baum-Connes assembly map for the groupoid $X\rtimes \Gamma$, $*=0,1$.
\end{theorem}

\bigskip

More precisely, we we have the six-term exact sequence which can be written as ($*=0$ and $*=1$)

\begin{displaymath}\label{Figure1}
\xymatrixcolsep{1pc}\xymatrix{
RK_{*} (X\rtimes \Gamma) \ar[rr]   & \stackrel{^{\mu^{BC}_{*+1}}}{} & K_{*} (C(X)\rtimes_{\redg} \Gamma)   \ar[dl]^{}\\ & S_{*+1} (X\rtimes \Gamma)  \ar[ul]_{}  }
\end{displaymath}

 \begin{remark}\
 Note that when $X=\star$ is reduced to a point, we recover the classical analytic surgery sequence of Higson-Roe \cite{HR-I, HR-II} for the group $\Gamma$.
\end{remark}
An obvious important corollary is the following

\begin{corollary}
The groupoid $X\rtimes \Gamma$ satisfies the Baum-Connes conjecture if and only if the groups $S_*(X\rtimes \Gamma)$ are trivial.
\end{corollary}

Notice that $X\rtimes \Gamma$ satisfies the Baum-Connes conjecture if and only if the group $\Gamma$ satisfies the Baum-Connes conjecture with coefficients in the $C^*$-algebra $C(X)$. Therefore, for all the discrete groups  $\Gamma$ which satisfy the Baum-Connes conjecture with commutatives coefficients, we get the vanishing of the defect groups $S_*(X\rtimes \Gamma)$ for all compact $\Gamma$-spaces $X$ as above.
The proof of Theorem \ref{Surgeryseq} uses most of the results proved before as well as the Paschke duality isomorphism explained in Section \ref{AppendixC}.  Let us start with the following

\begin{proposition}\label{PaschkeSquare}
Assume that $\Gamma$ acts properly and cocompactly on the locally compact (Hausdorff) space $Z$ and let  $i: Z'\hookrightarrow Z$ be a closed $\Gamma$-invariant subspace as before. Then we have a commutative diagram of group homomorphisms ($i=0, 1\in \Z_2$):
 \[
\begin{CD}\label{MoritaCommutation}
K_i(Q^*_\Gamma (X; (Z', L^2Z'\otimes \ell^2\Gamma^\infty))) @> i^Q_{Z'\subset Z}   >> K_i(Q^*_\Gamma (X; (Z, L^2Z\otimes \ell^2\Gamma^\infty))) \\
@V\maP^{Z'}_* VV    @V\maP^Z_{*}VV \\
KK^{i+1}_\Gamma(Z', X) @> {i_{*}}>> KK_{\Gamma}^{i+1}(Z, X)\\
\end{CD}
\]
where the vertical maps $\maP^\bullet_*$ are the Paschke duality isomorphisms described in Theorem \ref{Paschkev2}.
\end{proposition}

\begin{proof}\
Let us treat the case $i=0$.
Consider the representation $\pi: C_0(Z)\rightarrow \maL_{C(X)}(L^2Z\otimes \ell^2\Gamma^\infty\otimes C(X))$ given by $\pi:=\pi_Z^\infty\otimes \id_{C(X)}$. Similarly we define
$\pi':C_0(Z')\rightarrow \maL_{C(X)}(L^2Z'\otimes \ell^2\Gamma^\infty\otimes C(X)) $.  Let $W:L^2Z'\otimes \ell^2\Gamma^\infty
\rightarrow L^2Z\otimes \ell^2\Gamma^\infty$ be a {Roe-Voiculescu} covering $\Gamma$-isometry for the inclusion map $i:Z'\rightarrow Z$. Then the map $i^Q_{Z'\subset Z}$ is induced by
$$
\Ad_{W}: Q^*_\Gamma (X; (Z', L^2Z'\otimes \ell^2\Gamma^\infty)))\rightarrow Q^*_\Gamma (X; (Z, L^2Z\otimes \ell^2\Gamma^\infty)))
$$
Therefore, the composite map $\maP^{Z}_0\circ i^Q_{Z'\subset Z}$ is given for $P'\in \Proj(Q^*_\Gamma(X;(Z, L^2Z\otimes \ell^2\Gamma^\infty)))$
 by the formula
$$
\maP^{Z}_0\circ i^Q_{Z'\subset Z}([P'])= [(\pi, L^2Z\otimes \ell^2\Gamma^\infty\otimes C(X), 2(W\otimes \id_{C(X)})P'(W^*\otimes \id_{C(X)}) -\id)].$$
Let $i^*: C_0(Z)\rightarrow C_0(Z')$ be the restriction map. Then the composite map $i_{*}\circ \maP^{Z'}_0$ is given by
$$
i_{*}\circ \maP^{Z'}_0([P'])= [(\pi'\circ i^*, L^2Z'\otimes \ell^2\Gamma^\infty\otimes C(X), 2P'-\id)]
$$
Consider the projection $S= WW^*$ which induces a decomposition of $L^2Z\otimes \ell^2\Gamma^\infty$ into a direct sum $H_1\oplus H_2$, where $H_1:= \Im(S)$ and $H_2:= \Im(\id- S)$.
There is a corresponding decomposition of the Hilbert module
$L^2Z\otimes \ell^2\Gamma^\infty\otimes C(X)$ into orthocomplemented submodules:
$$
L^2Z\otimes \ell^2\Gamma^\infty\otimes C(X)= E_1\oplus E_2
$$
where $E_1=\Im(S\otimes \id_{C(X)})$ and $E_2= \Im((\id-S)\otimes \id_{C(X)})$. In particular, the operator $ W\otimes \id_{C(X)}: L^2Z'\otimes \ell^2\Gamma^\infty\otimes C(X)\rightarrow E_1$ is then a unitary isomorphism. Now, the operator
$$
2(W\otimes \id_{C(X)})P'(W^*\otimes \id_{C(X)})-\id_{L^2Z\otimes \ell^2\Gamma^\infty\otimes C(X)} \in \maL_{C(X)}(L^2Z\otimes \ell^2\Gamma^\infty\otimes C(X))
$$
can  be expressed in matrix form as follows:
$$
2(W\otimes \id_{C(X)})P'(W^*\otimes \id_{C(X)})-\id= \begin{bmatrix}
              (W\otimes \id_{C(X)})(2P'-\id)(W^*\otimes \id_{C(X)}) & 0\\ 0 & -\id_{E_2}
             \end{bmatrix}
$$
Therefore we have the following chain of equivalences in $KK_\Gamma^1(Z, X)$:
\begin{eqnarray*}
 &&(\pi'\circ i^*, L^2Z'\otimes \ell^2\Gamma^\infty\otimes C(X), 2P'-\id) \\
 &\sim&\left( (W\otimes \id_{C(X)})(\pi'\circ i^*)(\bullet) (W^*\otimes \id_{C(X)}), E_1, (W\otimes \id_{C(X)})(2P'-\id)(W^*\otimes \id_{C(X)}) \right)\; \\
 &\sim&( (S\otimes \id_{C(X)})\pi(\bullet) (S\otimes \id_{C(X)}), E_1, (W\otimes \id_{C(X)})(2P'-\id)(W^*\otimes \id_{C(X)})  )\quad \\
 &\sim&( (S\otimes \id_{C(X)})\pi(\bullet) (S\otimes \id_{C(X)})\oplus ((I-S)\otimes \id_{C(X)})\pi(\bullet) ((I-S)\otimes \id_{C(X)}),E_1\oplus E_2,\\
 && (W\otimes \id_{C(X)})(2P'-\id)(W^*\otimes \id_{C(X)})\oplus -\id_{E_2})) \\
 &&= (\pi, L^2Z\otimes \ell^2\Gamma^\infty\otimes C(X), 2(W\otimes \id_{C(X)}) P' (W^*\otimes \id_{C(X)})-\id_{L^Z\otimes \ell^2\Gamma^\infty\otimes C(X)})
\end{eqnarray*}
The first equivalence is induced by unitary conjugation, the second equivalence is due to Lemma \ref{KK_G-equiv} and the last equivalence corresponds to addition of a degenerate cycle. Hence, the proof is complete for $i=0$.  The proof for $i=1$ is similar and is omitted.
\end{proof}

\begin{corollary}
The Paschke  duality isomorphisms described in Theorem \ref{Paschkev2}, induce the two well defined isomorphisms
$$
\maP_*\; :\;  \lim_{\stackrel{\longrightarrow}{Z\subset \underline{E}\Gamma}} K_{*}\left(Q^*_\Gamma (X ; (Z, L^2Z\otimes \ell^2\Gamma^\infty))\right) \stackrel{\cong}{\longrightarrow} RK_{*+1}(X\rtimes\Gamma), \quad *=0, 1\in \Z_2.
$$
\end{corollary}

\begin{proposition}\label{Moritasquare}
 With the same assumptions as in Proposition \ref{PaschkeSquare}, we have a commutative diagram of group isomorphisms ($i\in \Z_2$):
\[
\begin{CD}\label{MoritaCommutation}
K_i(C^*_\Gamma (X; (Z', L^2Z'\otimes \ell^2\Gamma^\infty))) @> i^C_{Z'\subset Z}   >> K_i(C^*_\Gamma (X; (Z, L^2Z\otimes \ell^2\Gamma^\infty))) \\
@V\maM^{Z'}_* VV    @V\maM^Z_{*}VV \\
K_{i}(C(X)\rtimes_{\redg}\Gamma) @> \id >> K_{i}(C(X) \rtimes_{\redg}\Gamma)\\
\end{CD}
\]
Here $\maM^\bullet_*$ is the Morita isomorphism described in \cite{BenameurRoyI}\rm{[Section 2]}, applied to the transformation groupoid $G=X\rtimes \Gamma$ and for the cocompact proper $G$-spaces $X\times Z'$ and $X\times Z$.
\end{proposition}

\begin{proof}\
Let $W: L^2Z'\otimes \ell^2\Gamma^\infty \to L^2Z\otimes \ell^2\Gamma^\infty$ be a Roe covering $\Gamma$-isometry over the inclusion $i:Z'\hookrightarrow Z$. Then the map $\id_{C(X)}\otimes W$ sends $C(X)\otimes C_c(Z', \ell^2\Gamma^\infty)$ inside the Hilbert module   $L^2_{X\rtimes \Gamma} (X\times Z)\otimes \ell^2\Gamma^\infty$, see \cite{BenameurRoyI}\rm{[Section 2]}. Indeed, the isometry $W$ has finite propagation and sends any continuous compactly supported function from $C_c(Z', \ell^2\Gamma^\infty)$ to a $Z$-compactly supported class in $L^2(Z, \ell^2\Gamma^\infty)$. Moreover, we have for any $\alpha\in C(X)$ and $\rho\in C_c(Z', \ell^2\Gamma^\infty)$:
$$
\left<\alpha\otimes W\rho, \alpha\otimes W\rho\right>_{C(X)\rtimes\Gamma} (x, \gamma)=[\bar\alpha (\gamma^*\alpha)] (x) \left< W\rho, \gamma^*(W\rho)  \right>_{L^2Z\otimes \ell^2\Gamma^\infty}.
$$
But $W$ is a $\Gamma$-equivariant isometry from $L^2Z'\otimes \ell^2\Gamma^\infty$ to $L^2Z\otimes \ell^2\Gamma^\infty$, so we deduce
$$
\left<\alpha\otimes W\rho, \alpha\otimes W\rho\right>_{C(X)\rtimes\Gamma} (x, \gamma) = [\bar\alpha (\gamma^*\alpha)] (x) \left< \rho, \gamma^*(\rho)  \right>_{L^2Z\otimes \ell^2\Gamma^\infty}.
$$
This shows that
$$
\left<\alpha\otimes W\rho, \alpha\otimes W\rho\right>_{C(X)\rtimes\Gamma} = \left<\alpha\otimes \rho, \alpha\otimes\rho\right>_{C(X)\rtimes\Gamma}.
$$
Moreover the induced isometry, denoted
$$
{\what W}: L^2_{X\rtimes \Gamma} (X\times Z')\otimes \ell^2\Gamma^\infty\rightarrow L^2_{X\rtimes \Gamma} (X\times Z)\otimes \ell^2\Gamma^\infty,
$$
is adjointable. In particular, it is $C_c(X\times \Gamma)$-linear but this can be checked immediately, as for $\xi\in C_c(X\times Z', \ell^2\Gamma^\infty)$ and $f\in C_c(X\times \Gamma)$ one has
$$
{\what W} (\xi f)  =  {\what W} \left( \sum_{\gamma\in \Gamma} f(\gamma) (\gamma \xi)\right)
 =  \sum_{\gamma\in \Gamma} f(\gamma) (\id_{C(X)}\otimes W) (\gamma\xi)
 =  \sum_{\gamma\in \Gamma} f(\gamma) \gamma(\id_{C(X)}\otimes W) \xi = {\what W} (\xi) f.
$$
In a similar computation, one sees that
$$
(\id_{C(X)} \otimes W) \; \circ \; \Phi^{X\times Z'} \; = \;  \Phi^{X\times Z}  \; \circ \; ({\what W} \otimes_\lambda \id_{\maE_{X\rtimes \Gamma}}).
$$
We have denoted here $\Phi^{X\times Z}$ the isomorphism of \cite{BenameurRoyI}[Proposition 2.16]  for the proper cocompact space $X\times Z$, and similarly for $X\times Z'$, but that we have tensored with the identity of the $\Gamma$-representation $\ell^2\Gamma^\infty$. {{So, we have
$$
\Phi^{X\times Z} \, : \, \left(L^2_{X\rtimes \Gamma} (X\times Z)\otimes \ell^2\Gamma^\infty\right) \otimes_\lambda \maE_{X\rtimes \Gamma} \longrightarrow C(X)\otimes L^2Z\otimes \ell^2\Gamma^\infty,
$$
and similarly for $\Phi^{X\times Z'}$.}}

Notice  now that the Morita isomorphism for $X\times Z$ as {recalled} in Section \ref{GammaFamilies}, and tensored with the identity of $\ell^2\Gamma^\infty$, is given as Kasparov product with the class in
$$
KK (C^*_\Gamma (X; (Z, L^2Z\otimes \ell^2\Gamma^\infty)), C(X)\rtimes_{\redg} \Gamma)
$$
of the Kasparov cycle
$$
((\Phi^{X\times Z}_*)^{-1}, L^2_{X\rtimes \Gamma} (X\times Z)\otimes \ell^2\Gamma^\infty ,  0).
$$
See again \cite{BenameurRoyI} for the details. Indeed $\Phi^{X\times Z}_*$ is the map $T\mapsto \Phi^{X,Z} \circ (T\otimes_{\lambda} \id) \circ (\Phi^{X,Z})^{-1}$ which induces an isomorphism from the compact operators of $L^2_{X\rtimes \Gamma} (X\times Z)\otimes \ell^2\Gamma^\infty$ to the $C^*$-algebra $C^*_\Gamma (X; (Z, L^2Z\otimes \ell^2\Gamma^\infty))$. The same construction works for $Z'$ in place of $Z$.

On the other hand, the Kasparov product $i^C_{Z'\subset Z} \otimes \maM^Z_{*}$ is represented by the cycle
$$
(L^2_{X\rtimes \Gamma} (X\times Z)\otimes \ell^2\Gamma^\infty , (\Phi^{X\times Z}_*)^{-1} \circ Ad_{\id_{C(X)}\otimes W}, 0).
$$
Using the isometry ${\what W}$ we also see that the class $\maM^{Z'}_{*}$ can be represented by the cycle
$$
(Ad_{\what W}\circ (\Phi^{X\times Z'}_*)^{-1}, {\what W} L^2_{X\rtimes \Gamma} (X\times Z')\otimes \ell^2\Gamma^\infty ,  0).
$$
which in turn coincides with the cycle
$$
((\Phi^{X\times Z}_*)^{-1}\circ Ad_{\id_{C(X)}\otimes W}, {\what W} L^2_{X\rtimes \Gamma} (X\times Z')\otimes \ell^2\Gamma^\infty , 0).
$$
It thus remains to show that this latter cycle is equivalent to the cycle
$$
((\Phi^{X\times Z}_*)^{-1} \circ Ad_{\id_{C(X)}\otimes W}, L^2_{X\rtimes \Gamma} (X\times Z)\otimes \ell^2\Gamma^\infty ,  0).
$$
Notice though that for any $T\in C^*_\Gamma (X; (Z', L^2Z'\otimes \ell^2\Gamma^\infty))$, and any $\xi$ in the range of the projection $\id - {\what W}{\what W}^*$ on the Hilbert module $L^2_{X\rtimes \Gamma} (X\times Z')\otimes \ell^2\Gamma^\infty$ over $C(X)\rtimes_{\redg}\Gamma$, we have
$$
\left((\Phi^{X\times Z}_*)^{-1}\circ Ad_{\id_{C(X)}\otimes W}\right) (T) (\xi) = 0,
$$
since $\xi\in \Ker ({\what W}^*)$ and by the relation
$$
(\Phi^{X\times Z}_*)^{-1} \circ Ad_{\id_{C(X)}\otimes W} \; = \; Ad_{\what W} \circ (\Phi^{X\times Z'}_*)^{-1}.
$$
Therefore, the cycle $((\Phi^{X\times Z}_*)^{-1} \circ Ad_{\id_{C(X)}\otimes W}, L^2_{X\rtimes \Gamma} (X\times Z)\otimes \ell^2\Gamma^\infty ,  0)$ is
indeed equivalent to the cycle $((\Phi^{X\times Z}_*)^{-1}\circ Ad_{\id_{C(X)}\otimes W}, {\what W} L^2_{X\rtimes \Gamma} (X\times Z')\otimes \ell^2\Gamma^\infty ,  0)$.
\end{proof}

To sum up, we have proved the following

\begin{proposition}\label{functorial1} With our notations, we have a commutative cube (for $i\in \Z_2$), i.e. all the square faces commute:\\\newline
\hspace{-4.5cm}
\begin{tiny}

\begin{tikzcd}[row sep=tiny, column sep=tiny]\hspace{-3cm}
& K_i(Q^*_\Gamma (X; (Z', L^2Z'\otimes \ell^2\Gamma^\infty))) \arrow[dl] \arrow[rr] \arrow[dd] & & K_i(Q^*_\Gamma (X; (Z, L^2Z\otimes \ell^2\Gamma^\infty))) \arrow[dl] \arrow[dd] \\
K_{i+1}(C^*_\Gamma (X; (Z', L^2Z'\otimes \ell^2\Gamma^\infty))) \arrow[rr, crossing over] \arrow[dd] & &  K_{i+1}(C^*_\Gamma (X; (Z, L^2Z\otimes \ell^2\Gamma^\infty)))\\
& KK^{i+1}_\Gamma (Z',X) \arrow[dl] \arrow[rr] & & KK^{i+1}_\Gamma (Z, X) \arrow[dl] \\
K_{i+1}(C(X) \rtimes_{\redg}\Gamma) \arrow[rr] & & K_{i+1}(C(X) \rtimes_{\redg}\Gamma) \arrow[from=uu, crossing over]\\
\end{tikzcd}
\end{tiny}
\end{proposition}

\begin{proof}
We have  shown in Proposition \ref{Moritasquare} that the front vertical square commutes with all arrows as isomorphisms and with the bottom horizontal morphism being the identity. The top horizontal square commutes, because we may use the same {Roe-Voiculescu} covering $\Gamma$-isometry over the inclusion $i: Z'\hookrightarrow Z$ in order to define $i_*^C$ and $i_*^Q$.  The back vertical square commutes by Proposition \ref{PaschkeSquare}. The bottom horizontal square is commutative as a classical result from the construction of the Baum-Connes map using the Michschenko idempotent and not depending on the choice of a particular cut-off function, see for instance \cite{TuHyper}. The two remaining squares, say the left and right side squares,  commute for any \'etale groupoid $G$ by \cite{BenameurRoyI}[Theorem 3.3].
\end{proof}

\medskip

Theorem  \ref{Surgeryseq} is now essentially proved, but  let us finish this section by explaining the main steps of this proof for the sake of completeness.

\begin{proof}(of Theorem \ref{Surgeryseq})\
For any proper cocompact $\Gamma$-space $Z$, we have the six-term exact sequence  ($i\in \Z_2$):

\begin{displaymath}\label{Figure1}
\xymatrixcolsep{1pc}\xymatrix{
K_i (Q^*_\Gamma (X; (Z, L^2Z\otimes \ell^2\Gamma^\infty))) \ar[rr]^{\partial}  &  & K_{i+1} (C^*_\Gamma (X; (Z, L^2Z\otimes \ell^2\Gamma^\infty)))  \ar[dl]^{}\\ & K_{i+1} (D^*_\Gamma (X; (Z, L^2Z\otimes \ell^2\Gamma^\infty)))  \ar[ul]_{}  }
\end{displaymath}
which is associated with the short exact sequence of $C^*$-algebras
$$
0\to C^*_\Gamma (X; (Z, L^2Z\otimes \ell^2\Gamma^\infty)) \hookrightarrow D^*_\Gamma (X; (Z, L^2Z\otimes \ell^2\Gamma^\infty)) \rightarrow Q^*_\Gamma (X; (Z, L^2Z\otimes \ell^2\Gamma^\infty))\to 0.
$$
According to the Paschke isomorphism of Theorem \ref{Paschkev2} and the Morita isomorphism we deduce a six term exact sequence:

\begin{displaymath}\label{Figure1}
\xymatrixcolsep{1pc}\xymatrix{
KK^{i+1}_\Gamma (Z; X) \ar[rr]   & \stackrel{^{\mu^{BC}_{i+1}}}{} & K_{i+1} (C(X)\rtimes_{\redg} \Gamma)   \ar[dl]^{}\\ & K_{i+1} (D^*_\Gamma (X; (Z, L^2Z\otimes \ell^2\Gamma^\infty)))  \ar[ul]_{}  }
\end{displaymath}
where one has to replace the morphisms by their compositions with the appropriate isomorphisms (Morita or Paschke). We have for instance shown in \cite{BenameurRoyI}[Section 3]  that the composite map of the Boundary map $\partial$ with the inverse of the Paschke map is nothing but the Baum-Connes map $\mu^{BC}$ for the $\Gamma$-space $Z$.

By using Proposition \ref{functorial1}, we see that the natural map from the corresponding six-term exact sequence associated with a closed $\Gamma$-subspace $Z'$ of $Z$ to that of the space $Z$ is morphism of six-term exact sequences. Hence we may take the direct limit of each of the components of these exact sequences with respect to the inductive limit of all closed cocompact $\Gamma$-subspaces of the (locally compact) classifying space ${\underline E}\Gamma$ for proper $\Gamma$-actions. This yields precisely the universal Higson-Roe six-term exact sequence associated with the transformation groupoid $X\rtimes \Gamma$ as stated in  Theorem \ref{Surgeryseq}.

\end{proof}

\appendix

\section{Covering isometries}\label{Covering}

This appendix reviews  some standard constructions of isometries associated with coarse maps between proper metric spaces. {{In this appendix, Our metric $\Gamma$-spaces will be proper but not necessarily cocompact and their metrics will always be assumed to be $\Gamma$-invariant.}} We start  with the notion of Roe covering isometry.

\begin{definition}\label{VRiso1}\cite{HRbook}\ Suppose that $(Z',d')$ and $(Z,d)$ are proper metric  $\Gamma$-spaces. Let $f:Z'\rightarrow Z$ be a {{coarse map}}.
An  isometry $W: L^2Z'\otimes \ell^2\Gamma^\infty \longrightarrow L^2Z\otimes \ell^2\Gamma^\infty$ is called a
\emph{Roe covering $\Gamma$-isometry} for $f$, if it is $\Gamma$-equivariant and has finite propagation with respect to $f$, i.e. there exists an $R>0$ such that
$$
\pi_{Z}^\infty (\phi)W\pi_{Z'}^\infty (\phi')=0, \quad \forall \phi\in C_c(Z)\text{ and }\phi'\in C_c(Z') \text{ with } d_{Z}(\supp(\phi),f(\supp(\phi')))>R.
$$
 In this case the \emph{propagation} of $W$, denoted $\Prop(W)$, is the least constant $R$ satisfying the above condition.
\end{definition}

{{As we shall see below, every  $\Gamma$-equivariant coarse map $f:Z'\rightarrow Z$ admits a Roe covering $\Gamma$-isometry.  In fact, only a coarse version of the $\Gamma$-equivariance of $f$ is needed. }}

{{\begin{definition}
Suppose that $(Z',d')$ and $(Z,d)$ are proper metric  $\Gamma$-spaces. A coarse map $f:Z'\rightarrow Z$  is coarsely equivariant if  there exists a constant $M\geq 0$ such that
$$
d(f(gz'), gf(z')) \leq M, \quad \forall g\in \Gamma \text{ and }z'\in Z'.
$$
\end{definition}}}

\begin{lemma}\label{Roeiso}
Given a coarse and {{coarsely equivariant}}  map $f:Z'\rightarrow Z$ between the proper  $\Gamma$-spaces, there always exist Roe covering $\Gamma$-isometries  for $f$.
\end{lemma}

\begin{proof}\
By the classical construction of (non-equivariant) covering isometries, see for instance \cite{HRbook}[Chapter 6], there always exists such a finite propagation isometry
$V: L^2Z'\rightarrow L^2Z$. We use a  cut-off function $c\in C(Z')$ together with the averaging procedure explained in {\cite{GWY17}, Appendix A}  to get a $\Gamma$-invariant isometry
$$
W : \ell^2\Gamma^\infty\otimes L^2Z'\rightarrow \ell^2\Gamma^\infty\otimes L^2Z.
$$
{{More precisely, recall that the cut-off function $c\in C_b(Z')$ satisfies the following conditions:
\begin{itemize}
\item For any compact subspace $K'$ in $Z'$, the set $\{g\in \Gamma, gK'\cap \Supp (c) \neq \emptyset\}$ is finite.
\item $\sum_{g\in \Gamma} g\, c = 1$.
\end{itemize}
The existence of  $c$  is ensured since $\Gamma$ acts properly  on $Z'$  (see for instance \cite{TuHyper})}}. Then we may apply the results of \cite{GWY17}, {Appendix A} for  $X=\{\bullet\}$, and deduce that there exists a family $(U_g)_{g\in \Gamma}$ of $\Gamma$-invariant operators on $\ell^2\Gamma^\infty$, such that
$$
U_g^*U_{g'} = \delta_{g, g'} \id \text{ and } \sum_{g\in \Gamma} U_g U_g^* = \id.
$$
Set then
$$
W := \sum_{g\in \Gamma} (U_g\otimes \id_{L^2Z})  E_g (\id_{\ell^2\Gamma^\infty} \otimes V\pi_{Z'}(\sqrt{c})) E'_{g^{-1}}.
$$
Here $E$ and $E'$ are the unitary representations of $\Gamma$ in $\ell^2\Gamma^\infty\otimes L^2Z$ and $\ell^2\Gamma^\infty\otimes L^2Z'$ respectively,
obtained by tensoring the unitary representations $X$ and $X'$ in $L^2Z$ and $L^2Z'$, with the {right} regular action on $\ell^2\Gamma$ tensored further with the identity of $\ell^2\N$. {{Notice that using the regular representation in $\ell^2\Gamma$ is not important here as the formula for $W$ doesn't depend on this choice of representation. The operator $W$ is then an isometry which can be written as
$$
W= \sum_{g\in \Gamma} \; U_g\otimes X_g V \pi_{Z'}(\sqrt{c}) X'_{g^{-1}}.
$$
It is then obvious, since $c$ is vertically compactly supported (support condition on $c$) and from the properness of the action of $\Gamma$ on $Z$ as well, that this sum is locally finite. More precisely, for any given compactly supported continuous function $\xi'\in C_c(Z')$, the support of  $\pi_{Z'}(\sqrt{c}) X'_{g^{-1}}\xi'$ is  only non-empty for a finite number of $g$'s, depending of course on the chosen $\xi'$. Since $V$ has finite propagation, }} the isometry {$W$} has finite propagation as well.
{{Moreover}}, since $d_Z$ and $d_{Z'}$ are $\Gamma$-invariant, the propagation of $W$ is estimated {{by the propagation of $V$ plus the propagation of the action of $\Gamma$}}.
Indeed, consider $\phi\in C_c(Z), \phi'\in C_c(Z')$ such that {{$d_Z(\supp(\phi), f(\supp(\phi'))) >\Prop V + M_1$, where $M_1$ satisfies:}}
{{
$$
d(f(gz'), gf(z')) \leq M_1,   \quad \forall g\in \Gamma \text{ and }z'\in Z',
$$
then we have:}}
\begin{eqnarray*}
\pi_Z^\infty(\phi)W\pi_{Z'}^\infty(\phi') &=& \sum_{g\in \Gamma}  \pi_Z^\infty(\phi) (U_g\otimes \id_{L^2Z} )  E_g (\id_{\ell^2\Gamma^\infty} \otimes V\pi_{Z'}(\sqrt{c})) E'_{g^{-1}} \pi_{Z'}^\infty(\phi')\\
&=& \sum_{g\in \Gamma}  (U_g\otimes \id_{L^2Z} ) \pi_Z^\infty(\phi)  E_g (\id_{\ell^2\Gamma^\infty} \otimes V\pi_{Z'}(\sqrt{c})) E'_{g^{-1}} \pi_{Z'}^\infty(\phi')\\
&=& \sum_{g\in \Gamma}  (U_g\otimes \id_{L^2Z} ) E_g \pi_Z^\infty(g^{-1}\phi)   (\id_{\ell^2\Gamma^\infty} \otimes V\pi_{Z'}(\sqrt{c})) \pi_{Z'}^\infty(g^{-1}\phi') E'_{g^{-1}}\\
&=& \sum_{g\in \Gamma}  (U_g\otimes \id_{L^2Z} ) E_g   (\id_{\ell^2\Gamma^\infty} \otimes (\pi_Z(g^{-1}\phi) V\pi_{Z'}(g^{-1}\phi') )\pi_{Z'}(\sqrt{c}))  E'_{g^{-1}}\\
&=& 0
\end{eqnarray*}
{{since for any $g\in \Gamma$ any $z\in \Supp (\phi)$ and any $z'\in \Supp (\phi')$, we have
\begin{multline*}
d_Z (g^{-1} z, f(g^{-1} z')) \geq d_Z (g^{-1} z, g^{-1}f(z')) - d_Z (g^{-1} f(z'), f(g^{-1} z')) \\= d_Z (z, f(z')) - d_Z (g^{-1} f(z'), f(g^{-1} z')) > (\Prop V + M_1) - M_1,
\end{multline*}}}
{{so that $
d_Z(\supp(g^{-1}\phi), f(\supp(g^{-1}\phi')) >  \Prop V,$
and hence the term $\pi_Z(g^{-1}\phi) V\pi_{Z'}(g^{-1}\phi')$ vanishes. }}Note that in the above computation we have also used the fact that the sum is finite. Thus we get $\Prop W\leq \Prop V$.
We now show that $W$ is a $\Gamma$-equivariant isometry, i.e. for any $h\in \Gamma$,
$$
W E'_h= E_hW
$$
This is shown in the following routine computation:
\begin{eqnarray*}
W E'_h &=&  \sum_{g\in \Gamma} (U_g\otimes \id_{L^2Z})  E_g (\id_{\ell^2\Gamma^\infty} \otimes V\pi_{Z'}(\sqrt{c})) E'_{g^{-1}}E'_h\\
&=& \sum_{g\in \Gamma} (U_g\otimes \id_{L^2Z})  E_g (\id_{\ell^2\Gamma^\infty} \otimes V\pi_{Z'}(\sqrt{c})) E'_{g^{-1}h}\\
&=& \sum_{g'\in \Gamma} (U_{hg'}\otimes \id_{L^2Z})  E_{hg'} (\id_{\ell^2\Gamma^\infty} \otimes V\pi_{Z'}(\sqrt{c})) E'_{(g')^{-1}}\\
&=&\sum_{g'\in \Gamma} ((U_{hg'}\otimes \id_{L^2Z})E_h ) E_{g'} (\id_{\ell^2\Gamma^\infty} \otimes V\pi_{Z'}(\sqrt{c})) E'_{(g')^{-1}}\\
&=&\sum_{g'\in \Gamma} E_h (U_{g'}\otimes \id_{L^2Z}) E_{g'} (\id_{\ell^2\Gamma^\infty} \otimes V\pi_{Z'}(\sqrt{c})) E'_{(g')^{-1}}\\
&=& E_h W
\end{eqnarray*}
where in the second last line we have used the equivariance property of the family of isometries $U_g, g\in \Gamma$ with respect to the right regular action $\rho$ on $\ell^2\Gamma$, i.e.
$$
U_g (\rho_h\otimes \id_{L^2Z})= (\rho_h\otimes \id_{L^2Z})U_{h^{-1}g} \text{ for } g, h\in \Gamma.
$$
 \end{proof}

We also used the more restrictive notion of {Roe-Voiculescu} covering $\Gamma$-isometry.

\begin{definition}\label{VRiso2}\
Let $f:Z'\rightarrow Z$ be a \emph{continuous} $\Gamma$-equivariant coarse map. A bounded operator
$W\in B( \ell^2\Gamma^\infty\otimes L^2Z', \; \ell^2\Gamma^\infty\otimes L^2Z)$
will be called a \emph{{Roe-Voiculescu} covering $\Gamma$-isometry} for $f$, if it satisfies the following properties:
\begin{enumerate}
\item $W$ is a $\Gamma$-equivariant isometry;
  \item $W$ has finite propagation with respect to $f$;
  \item\label{Voiculescu}
  For any $\phi\in C_0(Z)$, we have $W^*\pi^\infty_{Z}(\phi)W -\pi^\infty_{Z'}(\phi\circ f) \in \maK (\ell^2\Gamma^\infty\otimes L^2Z')$.
 \end{enumerate}
\end{definition}
So, a {Roe-Voiculescu} covering $\Gamma$-isometry for $f$ is a Roe covering $\Gamma$-isometry for $f$ which satisfies the extra condition \eqref{Voiculescu}.

\begin{lemma}\label{VR-isometry}
 For any continuous coarse {{and coarsely equivariant}} map $f:Z'\rightarrow Z$, there exists {Roe-Voiculescu} covering $\Gamma$-isometries for $f$.
 \end{lemma}
\begin{proof}\
The Hilbert space $ \ell^2\N\otimes \ell^2\Gamma^\infty \otimes L^2(Z)$ equipped with the representation $\id_{\ell^2\N}\otimes \pi_Z^\infty$ of $C_0(Z)$ is a
very-ample representation (i.e. a countably infinite direct sum of a fixed ample representation), so by \cite{HRbook}[Lemma 12.4.6],
there exists an isometry $V: \ell^2\Gamma^\infty\otimes L^2(Z')\rightarrow \ell^2\N\otimes \ell^2\Gamma^\infty \otimes L^2(Z)$
which has finite propagation and satisfies the condition:
$$
V^*(\id_{\ell^2\N}\otimes \id_{\ell^2\Gamma^\infty}\otimes \pi_Z (\phi)) V - (\id_{\ell^2\Gamma^\infty}\otimes \pi_{Z'} (\phi\circ f))\;  \in\;  \maK({\ell^2\Gamma^\infty}\otimes L^2(Z')),  \quad \forall \phi\in C_0(Z).
$$
Again, this isometry is not $\Gamma$-equivariant in general. We first compose $V$ with a unitary $u_\infty:\ell^2\N\otimes \ell^2\N\rightarrow \ell^2\N$ to
get back from $\ell^2\N \otimes \ell^2\Gamma^\infty$ to $\ell^2\Gamma^\infty$, and obtain the isometry:
$$
\bar{V}:= (u_\infty\otimes\id_{\ell^2\Gamma}\otimes \id_{L^2Z})\circ V: \ell^2\Gamma^\infty \otimes L^2(Z')\rightarrow \ell^2\Gamma^\infty \otimes L^2(Z).
$$
Since $u_\infty\otimes \id_{\ell^2\Gamma}\otimes \id_{L^2Z}$ commutes with the representation $\id_{\ell^2\Gamma^\infty}\otimes \pi_Z$, it is straightforward to check that
$\bar{V}$ has propagation bounded above by $\Prop V$, and it satisfies the intertwining property:
$$
\bar{V}^*(\id_{\ell^2\Gamma^\infty}\otimes \pi_{Z}(\phi))\bar{V}-(\id_{\ell^2\Gamma^\infty}\otimes\pi_{Z'}(\phi\circ f)) \in \maK(\ell^2\Gamma^\infty\otimes L^2Z'), \quad \forall \phi\in C_0(Z).
$$
Finally, {{using a {cut-off} function $c\in C_b(Z')$ as in the proof of Lemma \ref{Roeiso}, we replace $\bar{V}$ by a $\Gamma$-invariant isometry $W$ in the same way as in that proof}}, i.e.
$$
W := \sum_{g\in \Gamma} (U_g\otimes \id_{L^2Z})  E_g \bar{V}(\id_{\ell^2\Gamma^\infty} \otimes \pi_{Z'}(\sqrt{c})) E'_{g^{-1}}.
$$
That the isometry $W$ is a Roe covering $\Gamma$-isometry for $f$ is clear, see again the proof of Lemma  \ref{Roeiso}.  It remains to verify Condition \eqref{Voiculescu}.
We have for $\phi\in C_0(Z)$:
\begin{eqnarray*}
W^*\pi_Z^\infty(\phi)W &=&  [\sum_{g\in \Gamma} (U_g\otimes \id_{L^2Z})  E_g \bar{V}(\id_{\ell^2\Gamma^\infty} \otimes \pi_{Z'}(\sqrt{c})) E'_{g^{-1}}]^*\\
&& \times(\id_{\ell^2\Gamma^\infty}\otimes \pi_{Z}(\phi))[\sum_{h\in \Gamma} (U_h\otimes \id_{L^2Z})  E_h \bar{V}(\id_{\ell^2\Gamma^\infty} \otimes \pi_{Z'}(\sqrt{c})) E'_{h^{-1}}]\\
&=& [\sum_{g\in \Gamma} (E'_{g^{-1}})^*(\id_{\ell^2\Gamma^\infty} \otimes \pi_{Z'}(\sqrt{c}))\bar{V}^*E^*_g (U^*_g\otimes \id_{L^2Z})]\\
&  & \times(\id_{\ell^2\Gamma^\infty}\otimes \pi_{Z}(\phi))[\sum_{h\in \Gamma} (U_h\otimes \id_{L^2Z})  E_h \bar{V}(\id_{\ell^2\Gamma^\infty} \otimes \pi_{Z'}(\sqrt{c})) E'_{h^{-1}}\\
&=& \sum_{g\in\Gamma} (E'_{g^{-1}})^*(\id_{\ell^2\Gamma^\infty} \otimes \pi_{Z'}(\sqrt{c}))\bar{V}^*E^*_g (U^*_g\otimes \id_{L^2Z})\\
&& \times(\id_{\ell^2\Gamma^\infty}\otimes \pi_{Z}(\phi)) (U_g\otimes \id_{L^2Z})  E_g \bar{V}(\id_{\ell^2\Gamma^\infty} \otimes \pi_{Z'}(\sqrt{c})) E'_{g^{-1}}
\end{eqnarray*}
which gives
\begin{eqnarray*}
W^*\pi_Z^\infty(\phi)W
& = & \sum_{g\in\Gamma} (E'_{g^{-1}})^*(\id_{\ell^2\Gamma^\infty} \otimes \pi_{Z'}(\sqrt{c}))\bar{V}^*(E^*_g (\id_{\ell^2\Gamma^\infty}\otimes \pi_{Z}(\phi)) E_g) \bar{V}(\id_{\ell^2\Gamma^\infty} \otimes \pi_{Z'}(\sqrt{c})) E'_{g^{-1}}\\
&=& \sum_{g\in\Gamma} (E'_{g^{-1}})^*(\id_{\ell^2\Gamma^\infty} \otimes \pi_{Z'}(\sqrt{c}))\bar{V}^* (\id_{\ell^2\Gamma^\infty}\otimes \pi_{Z}(g^{-1}\phi)) \bar{V}(\id_{\ell^2\Gamma^\infty} \otimes \pi_{Z'}(\sqrt{c})) E'_{g^{-1}}\\
& \sim & \sum_{g\in\Gamma} (E'_{g^{-1}})^*(\id_{\ell^2\Gamma^\infty} \otimes \pi_{Z'}(\sqrt{c})) (\id_{\ell^2\Gamma^\infty}\otimes \pi_{Z'}(g^{-1}\phi\circ f)) (\id_{\ell^2\Gamma^\infty} \otimes \pi_{Z'}(\sqrt{c})) E'_{g^{-1}}\\
&=& \sum_{g\in\Gamma} (E'_{g^{-1}})^*(\id_{\ell^2\Gamma^\infty} \otimes \pi_{Z'}({c\cdot}(g^{-1}\phi\circ f))  E'_{g^{-1}}\\
&=& (\id_{\ell^2\Gamma^\infty}\otimes \pi_{Z'}(\phi\circ f){)}\sum_{g\in \Gamma}  (\id_{\ell^2\Gamma^\infty}\otimes \pi_{Z'}(g c))\\
&=& (\id_{\ell^2\Gamma^\infty}\otimes \pi_{Z'}(\phi\circ f)) \\
&=&\pi_{Z'}^\infty(\phi\circ f)
\end{eqnarray*}

%
%
\end{proof}

\section{Kasparov's homological lemma}

We have used a classical result due to Kasparov but in the context of  groupoid equivariant $KK$-theory (see \cite{Kasparov1}[Section 7, Lemma 2]). This is a classical result, well-known to experts, but we add it for completeness. So, $G$ is here an \'etale Hausdorff locally compact groupoid and we assume for simplicity that the space of units $X$ is compact.
 Let $A, B$ be unital nuclear separable $G$-algebras and suppose that $\maH_G$ is an absorbing $G$-Hilbert $B$-module, i.e. any countably generated $G$-Hilbert $B$-module is isometric to an
orthocomplemented $G$-submodule
of $\maH_G$. Denote by $E_G(A, B)$ the set of triples $(\pi, \maH_G, F)$, where $\pi: A\rightarrow \maL_B(\maH_G)$ is a $G$-equivariant
representation and $F$ is a $G$-invariant operator modulo {compact oeprators}, which satisfies the usual conditions of a $KK$-cycle:
$$
[F, \pi(a)], \quad \pi(a)(F^2-I),\quad \pi(a)(F-F^*) \in \maK(\maH_G),\quad\quad \forall a\in A
$$
The equivalence relation on such triples is generated by unitary equivalence by $G$-invariant unitaries, addition of degenerate cycles and
operator homotopy.

\begin{lemma}[Kasparov's homological $KK_G$-equivalence lemma ]\label{KK_G-equiv}
Let $A, B$ be nuclear separable $G$-algebras and consider a $G$-equivariant representation
$\pi_1: A\rightarrow \maL_B(\maH_G)$. If $(\pi_2, \maH_G, F) \in E_G(A, B)$
and there exists a $G$-invariant {unitary} $S\in \maL(\maH_G)$ such that
$$
S^*\pi_1(a)S- \pi_2(a) \in \maK(\maH_G)
$$
then the cycles $(\pi_2, \maH_G, F)$ and $(\pi_1, \maH_G, SFS^*)$ define the same $KK_G$-class.
\end{lemma}

\begin{proof}
The proof in \cite{Kasparov1}[pp 561-562] can be used to prove the statement as follows. Consider a pair $(\phi, P)$, where $\phi: A\rightarrow \maL_B(\maH_G)$ is a $G$-equivariant representation and
$P\in \maL_B(\maH_G)$ is a {$G$-invariant operator modulo {compact operators}} such that we have:
$$
[P, \phi(a)], \quad \phi(a)(P^2-P),\quad \phi(a)(P-P^*) \in \maK(\maH_G)\quad\quad \forall a\in A
$$

Any such pair gives rise to a $KK_G$-class given by $(\phi, \maH_G, 2P-\id) $ and conversely a triple $(\pi, \maH_G,F)$ representing a $KK_G$-class gives rise to a pair $(\pi, (F+1)/2)$ satisfying the above conditions. Two such pairs $(\phi_1, P_1)$ and $(\phi_2, P_2)$ are called homological if $P_1\phi_1(a)\sim P_2\phi_2(a), \forall a\in A$.
Homological pairs give rise to the same $KK_G$-class by the proof of \cite{Kasparov1}, Section 7, Lemma 2. {Indeed, an explicit operator homotopy between the $KK_G$-classes
induced by $(\phi_1,P_1)\oplus (\phi_2, 0)$ and $(\phi_1,0)\oplus(\phi_2,P_2)$ is induced by the operator homotopy for pairs:
$$
\left(\begin{bmatrix} \phi_1 & 0\\ 0 & \phi_2 \end{bmatrix},\frac{1}{1+t^2} \begin{bmatrix} P_1   & tP_1P_2\\ tP_2P_1 & t^2P_2  \end{bmatrix}\right) \quad \text{ for } \quad 0\leq t \leq \infty
$$}
{Note that the direct sum $\maH_G\oplus \maH_G$ is endowed with the diagonal $G$-action, with respect to which the above operator matrix is $G$-equivariant up to {compact operators}. Thus the operator
homotopy goes through well-defined $KK_G$-classes. We now claim that $(\pi_1,\maH_G, SFS^*)$ and $(S\pi_2 S^*, \maH_G, SFS^*)$ are in the same $KK_G$-class.
Indeed, if we define:
$$
P_1= \frac{1}{2}(SFS^*+I)= P_2,\quad \phi_1= \pi_1,\quad \phi_2= S\pi_2S^*,
$$
then the pairs $(P_1,\phi_1)$ and $(P_2, \phi_2)$ are homological, {and} then give rise to the $KK_G$-classes
of $(\pi_1, \maH_G, SFS^*)$ and $(S\pi_2 S^*, \maH_G, SFS^*)$, respectively.
We then have a chain of $KK_G$-equivalences:
$$
(\pi_2, \maH_G, F)\sim (S\pi_2S^*, \maH_G, SFS^*)\sim (\pi_1,\maH_G, SFS^*) 
$$}


\end{proof}

\end{document}